\documentclass[12pt]{amsart}
\usepackage{amsmath}
\usepackage{amssymb}
\usepackage{tabularx}
\usepackage{enumerate}
\usepackage{graphicx}
\usepackage{texdraw}
\usepackage{color}

\usepackage{pgfplots}
\usepackage{graphics}
\usepackage{graphicx}
\usepackage{subfigure}
\usepackage{mathrsfs}

\usepackage{amsfonts,amssymb,amsmath}
\usepackage{epsfig}
\usepackage{bm}

\newtheorem{Theorem}{Theorem}[section]
\newtheorem{Lemma}[Theorem]{Lemma}

\newtheorem{example}{Example}[section]
\newtheorem{Definition}{Definition}[section]

\newcommand{\R}{\mathbb{R}}

\newcommand{\V}{\mathbf{v}}
\newcommand{\U}{\mathbf{u}}
\newcommand{\F}{\mathbf{f}}
\newcommand{\G}{\mathbf{g}}

\newcommand{\W}{\mathbf{w}}
\newcommand{\Z}{\mathbf{z}}
\numberwithin{equation}{section}
\numberwithin{figure}{section}

\begin{document}
	
	\title[]{Lotka-Volterra Competition Models on Finite Graphs}
	
\author[Y. Hu]{Yuanyang Hu$^1$}
\author[C. Lei]{Chengxia Lei$^2$}
\thanks{$^1$ School of Mathematics and Statistics,
	Henan University, Kaifeng, Henan 475004, P. R. China.}
\thanks{$^2$School of Mathematics and Statistics, Jiangsu Normal University,
	Xuzhou, 221116, Jiangsu Province, China.}

\thanks{{\bf Emails:} {\sf yuanyhu@mail.ustc.edu.cn} (Y. Hu).}
\thanks{{\bf Emails:} {\sf leichengxia001@163.com(C. Lei)}}
\thanks{Y. Hu was partially supported by NSF of China (No. 12201184), National Natural Science Foundation of
	He' nan Province of China (No. 222300420416) and China Postdoctoral Science Foundation (No. 2022M711045).}
\thanks{C. Lei was partially supported by NSF of China (No. 11801232, 11971454, 12271486), the NSF of Jiangsu Province (No. BK20180999) and the Foundation of Jiangsu Normal University (No. 17XLR008).}
\date{\today}	
	
	\begin{abstract}
		In this paper, we study three two competing species Lotka-Volterra competition models on finite connected graphs, with Dirichlet, Neumann or no boundary conditions. We get that when time goes to infinity,  either one specie extincts while the other becomes surviving or both competing species coexist, which depend crucially on the strengh of species' competitiveness and the size of the initial population under the Neumann boundary condition and the condition that there is no boundary condition. One of our results partially answer a question posed by Slav\'{\i}k in [SIAM J. Appl. Dyn. Syst., 19 (2020)]. The critical techiniques in the proof of our main results are upper and lower solutions method, which are developed for weakly coupled parabolic systems on finite graphs in this article.
	\end{abstract}
	

	\maketitle
	
	\section{Introduction}
Recently, increasing efforts have been devoted to the analysis of the Lotka-Volterra models on graphs, see, for example, \cite{S, S2, SC} and the references therein.
Suppose there are finite discrete patches $P=\{x_1,x_2,\cdots,x_n\}$ and the species can migrate between patches $x_i$ and $x_j$. In \cite{S},
Slav\'ik studied the following Lotka-Volterra competition model on a finite graph 
\begin{equation*}\label{20220904-2}
\left\{\begin{array}{lll}
u_{t}=d_{1}\sum\limits_{y \in P}(u(y,t)-u(x,t))+u(a_1-b_1u-c_1 v), &x \in P,&t>0, \\ 
v_{t}=d_{2}\sum\limits_{y \in P}(v(y,t)-v(x,t))+v(a_2-b_2u-c_2 v), &x\in P,&t>0,\\
u=u_{0}(x)\ge 0,~v=v_{0}(x) \ge 0, &x \in P,&t=0,
\end{array}\right.
\end{equation*}
where $a_{i},$ $b_{i}$, $c_{i}$ $>0$ and $d_i\ge 0$, $i=1,2$ are constants. They described the spatially homogeneous stationary states and their stability, discussed the
existence and number of spatially heterogeneous stationary states, and studied the asymptotic behavior
of solutions. In section 6 of \cite{S}, they raised an open problem as follows. For the generalized model 
 \begin{equation}\label{a}
 	\left\{\begin{array}{lll}
 		u_{t}=\sum\limits_{y \in P}(u(y,t)-u(x,t))d^{1}_{xy}+u(a_1-b_1u-c_1 v), &x \in P,&t>0, \\ 
 		v_{t}=\sum\limits_{y \in P}(v(y,t)-v(x,t))d^{2}_{xy}+v(a_2-b_2u-c_2 v), &x\in P,&t>0,\\
 		u=u_{0}(x)\ge 0,~v=v_{0}(x) \ge 0, &x \in P,&t=0,
 	\end{array}\right.
 \end{equation}
where the diffusion coefficients $d^{i}_{xy}\ge 0$ for all $x,y\in P$, $i=1,2$, it is unclear how to generalize Theorem 4.5 in \cite{S} to the problem \eqref{a}. We partially answer the question under the condition that $d^{i}_{xy}> 0$ for all $x,y\in P$, $i=1,2$. Furthermore, we give more results (see Theorem \ref{1,4} below).	
	
 We first introduce some concepts of graph theory that are often used in this paper.  A graph $G=G(V,E)$ we mean a finite set $V$ of vertices with a set $E$ of two-element subsets of $V$ (whose elements are called edges). We write $x\sim y$ ($x$ is adjacent to $y$) if $(x,y)\in E$. A graph is called simple if it has neither loops or multiple edges. A finite sequence $\{x_{k} \}_{k=0}^{n}$ of vertices on a graph is called a path if $x_{k}\sim x_{k+1}$ for all $k=0,1,\dots,n-1$. A graph $G=G(V,E)$ is said connected if, for any two vertices $x,y\in V$, there exists a path connecting $x$ and $y$, that is, a path $\{x_{k}\}_{k=0}^{n}$ such that $x_0=x$ and $x_n= y$. Throught this paper, all the graphs in our concern are assumed to connected. We define the weight function $\omega:~V\times V\to [0,+\infty) $ satisfying $\omega_{x y}=\omega_{yx }$, $x,y\in V$ and $\omega_{x y}=0$ if and only if $(x,y)\not\in E$. Let $\mu(x):V\to \mathbb{R}^{+}$ be a positive measure in the graph $G=(V,E)$. For a subgraph $\Omega$ of a graph $G=G({V},{E})$, the (vertex) boundary $\partial{\Omega}$ of $\Omega$ is the set of all vertices $z\in {V}$ not in $\Omega$ but adjacent to some vertex in $\Omega$, i.e., 
\begin{equation*}
	\partial{\Omega}:=\{z\in {V} \backslash \Omega | z\sim y ~\text{for some}~y\in \Omega \}.
\end{equation*} 
Denote $\bar{\Omega}$ a graph whose vertices and edges are in $\Omega$ and vertices in $\partial{\Omega}$.  Throught this paper, we always assume that ${\partial\Omega \not=\emptyset}$. The $\omega-$Laplacian $\Delta_{{V}}$ of a function $u:V\to\mathbb{R}$ on a graph $G=(V,E)$ is defined by 
\begin{equation*}
	\Delta_{\omega} u(x)=\Delta_{{V}} u(x):=\sum_{y \in V} (u(y)-u(x))\frac{\omega_{y x}}{\mu(x)},~x\in V,
\end{equation*}
and the $\omega-$Laplacian $\Delta_{\Omega}$ of $u$ on a subgraph $\Omega$ of $G$ is defined by 
\begin{equation}\label{}
	\Delta_{\omega} u(x)=\Delta_{\Omega} u(x):=\sum_{y \in \bar{\Omega}} (u(y)-u(x))\frac{\omega_{y x}}{\mu(x)},~x\in V.
\end{equation}
For more details of graph theoretic notions, we refer to \cite{CB} and \cite{HL}. 

Let $G=G(V,E)$ and $\bar{ \Omega}$ be finite weighted graphs. Let $\{{\omega}^{k}\}_{k=1}^{2}$ and $\{\tilde{\omega}^{k}\}_{k=1}^{2}$ be weight functions on $G$ and $\bar{\Omega}$ respectively. In this paper, we consider the following problems
	\begin{equation}\label{41}
	\left\{\begin{array}{lll}
	u_{t}-d_{1}\Delta_{V_1} u=u(a_1-b_1u-c_1 v), &x \in V, &t>0,\\ 
	v_{t}-d_{2}\Delta_{V_2} v=v(a_2-b_2u-c_2 v), &x \in V,& t>0,\\
	u=\tilde{u}_{0} ,~v=\tilde{v}_{0} , &x \in V,&t=0
	\end{array}\right.
	\end{equation}	
 and
 \begin{equation}\label{20220904-3}
\left\{\begin{array}{lll}
u_{t}-d_{1}\Delta_{\Omega_1} u=u(a_1-b_1u-c_1 v), & x \in \Omega, & t>0, \\
v_{t}-d_{2}\Delta_{\Omega_2} v=v(a_2-b_2u-c_2 v), & x \in \Omega, & t>0,\\
 {B_1}u={B_2}v=0, & x \in \partial \Omega, & t> 0,\\
u=u_{0} ,~v=v_{0} , &x \in {\bar{\Omega}}, &t=0,
\end{array}\right.
\end{equation}
Here $u,v$ represent the densities of two competing species, $a_i, b_i, c_i$ $(i=1,2)$ are positive constants, the initial functions $\tilde{u}_0$, $\tilde{v}_0$ satisfy $\tilde{u}_{0}\ge 0$, $\tilde{v}_{0}\ge 0$ on $V$, and $u_0$, $v_0$ satisfy $u_0\ge 0$, $v_0\ge 0 $ on $\bar{ \Omega}$,
 $$\Delta_{V_1} u(x):=\sum\limits_{y \in V} (u(y,t)-u(x,t))\frac{\omega^{1}_{y x}}{\tilde{\mu}^{1}(x)},~x\in V,$$
$$\Delta_{V_2} u(x):=\sum\limits_{y \in V} (u(y,t)-u(x,t))\frac{\omega^{2}_{y x}}{\tilde{\mu}^{2}(x)},~x\in V,$$
$$\Delta_{\Omega_1} u(x):=\sum\limits_{y \in \bar{ \Omega}} (u(y,t)-u(x,t))\frac{\tilde{\omega}^{1}_{y x}}{\mu^{1}(x)},~x\in \Omega,$$
$$\Delta_{\Omega_2} u(x):=\sum\limits_{y \in \bar{ \Omega}} (u(y,t)-u(x,t))\frac{\tilde{\omega}^{2}_{y x}}{\mu^{2}(x)},~x\in \Omega,$$
$\tilde{\mu}^{i}$ $(i=1,2)$ and ${\mu}^{i}$ $(i=1,2)$ are finite positive measures on $V$ and $\bar{ \Omega}$, respectively, and the operators $B_{i} (i=1,2)$  satisfy 
\begin{equation}\label{1,4}
	\begin{cases}
	{B_1}u(x,t)=\frac{\partial u}{\partial_{\Omega_1} n}(x):=\sum\limits_{y\in {\Omega}} (u(x,t)-u(y,t))\frac{\tilde{\omega}^{1}_{xy}}{{\mu}^{1}(x)}~\text{on}~\partial\Omega\times(0,+\infty),\\  {B_2}v(x,t)=\frac{\partial v}{\partial_{\Omega_2} n}(x):=\sum\limits_{y\in {\Omega}} (v(x,t)-v(y,t))\frac{\tilde{\omega}^{2}_{xy}}{{\mu}^{2}(x)}~\text{on}~\partial\Omega\times(0,+\infty),
	\end{cases}
\end{equation}
or
\begin{equation}\label{1,5}
	B_{1}u\equiv u=0,~B_{2}v\equiv v=0~\text{on}~\partial\Omega\times(0,+\infty).
\end{equation}

	The steady-state problem to \eqref{41} and \eqref{20220904-3} are 
	\begin{equation}\label{s}
		\left\{\begin{array}{lll}
			-d_{1}\Delta_{V_{1}} u=u(a_1-b_1u-c_1 v), & x \in V,  \\
			-d_{2}\Delta_{V_{2}} v=v(a_2-b_2u-c_2 v), & x \in V
		.
		\end{array}\right.
	\end{equation}
	and
	\begin{equation}\label{ss}
	\left\{\begin{array}{lll}
	-d_{1}\Delta_{\Omega_{1}} u=u(a_1-b_1u-c_1 v), & x \in \Omega,  \\
	-d_{2}\Delta_{\Omega_{2}} v=v(a_2-b_2u-c_2 v), & x \in \Omega,  \\
	{B_1} u={B_2} v=0, & x \in \partial \Omega,
	\end{array}\right.
	\end{equation} respectively.
Clearly, the problems \eqref{41} and \eqref{20220904-3} has four constant steady-state solutions: $(0,0)$, $(\frac{a_1}{b_1},0)$, $(0,\frac{a_2}{c_2})$, $(\xi, \eta)$, where 
\begin{equation}\label{20220924}
\xi =\frac{a_1 c_2- a_2 c_1}{\triangle}, \eta =\frac{a_2 b_1- a_1 b_2}{\triangle}, \text{when}~ \triangle=b_1 c_2- b_2 c_1\not = 0.
\end{equation}

	We are now ready to describe the main results of this paper. 
	\begin{Theorem}\label{1.1}
		Let $\bar{ \Omega}$ be a finite graph. Then for any given functions $u_0 \ge 0$, $v_0 \ge 0$ on $\bar{ \Omega}$, \eqref{20220904-3} admits a unique solution $(u,v)$ defined for all $t>0$. Furthermore, under the condition that $u_0 \not \equiv 0 $ and $v_0 \not \equiv 0$, we have
		\begin{description}
			\item[(i)] if $B_{i}$ $(i=1,2)$ satisfy \eqref{1,4}, then $u(x,t)>0$ and $v(x,t)>0$ for $x\in \bar{ \Omega}$ and $t>0$;
				\item[(ii)] if $B_{i}$ $(i=1,2)$ satisfy \eqref{1,5}, then $u(x,t)>0$ and $v(x,t)>0$ for $x\in \Omega$ and $t>0$.
		\end{description} 
	\end{Theorem}

	\begin{Theorem}\label{m}
Let $\bar{ \Omega}$ be a finite graph. Suppose that the initial functions $u_{0},$ $v_{0}\not \equiv 0$.	Let $(u,v)$ be the unique solution to \eqref{20220904-3} with ${B_i}$ $(i=1,2)$ satisfying \eqref{1,4}. Then we have
	\begin{description}
		\item[(i)] If $\frac{a_1}{a_2}< \frac{b_1}{b_2}$ , $\frac{a_1}{a_2}< \frac{c_1}{c_2}$, then $\lim\limits_{t\to \infty} (u(x,t),v(x,t))=(0,\frac{a_2}{c_2})$;
		\item[(ii)] 	
		If $\frac{b_1}{b_2}<\frac{a_1}{a_2},$ $\frac{c_1}{c_2}<\frac{a_1}{a_2}$, then $\lim\limits_{t\to \infty} (u(x,t),v(x,t))=(\frac{a_1}{b_1},0)$ ;
		\item[(iii)] 	
		If $\frac{c_{1}}{c_{2}}<\frac{a_{1}}{a_{2}}<\frac{b_{1}}{b_{2}}$, then
		$\lim\limits_{t\to +\infty}\left( u(x,t),v(x,t) \right) =(\xi,\eta)$;
		\item[(iv)] Assume that $\frac{b_{1}}{b_{2}}<\frac{a_{1}}{a_{2}}<\frac{c_{1}}{c_{2}}$. Then $\lim\limits _{t \rightarrow \infty}(u(x, t), v(x, t))=\left(\frac{a_{1}}{b_{1}}, 0\right)$ provided that $\xi<u_{0}(x)<\frac{a_{1}}{b_{1}}$ and $0<v_{0}(x)<\eta$ on $\bar{ \Omega}$;
		$\lim\limits _{t \rightarrow \infty}(u(x, t), v(x, t))=\left(0,\frac{a_{2}}{c_{2}} \right)$ provided that $0<u_0<\xi$ and $\eta<v_0(x)< \frac{a_2}{c_2}$ on $\bar{ \Omega}$.  
		\end{description} 
Here, $\xi$ and $\eta$ are defined by \eqref{20220924}.
\end{Theorem}

By \cite{C}, the following eigenvalue problem 
\begin{equation}\label{451}
	\begin{cases}
		-\Delta_{\Omega_{i}} u=\lambda u~\text{in}~\Omega,\\
		u=0~\text{on}~\partial{\Omega}.
	\end{cases}
\end{equation}
admits a smallest eigenvalue $\lambda_{0,i}>0$ $(i=1,2)$.

By \cite[Theorem 5.5]{HL}, we know that if $a_{i}>\lambda_{0,i}d_{i}$, then the following elliptic problem
\begin{equation}\label{20220904-1}
\begin{cases}
-d_{i}\Delta_{\Omega_{i}} s_i= s_i\left(a_i-E_i s_i \right) ~\text{ in }~\Omega,\\
s_{i}=0~\text{ on }~\partial{\Omega}\\
\end{cases}
\end{equation}
 admits a unique positive solution $s_{i}(x)$ ($i=1,2$), where $E_1=b_1$ and $E_2=c_2.$

For convenience, we denote $\mathcal{M}=\{\tilde{\omega}^{1},\tilde{\omega}^{2},\tilde{\mu}^{1},\tilde{\mu}^{2}\}$.	
	\begin{Theorem}\label{m1}
Let $\bar{ \Omega}$ be a finite graph. Assume that $(u,v)$ is the solution to \eqref{20220904-3} with ${B}_{i}$ satisfying \eqref{1,5}. Then we have the following:
	\begin{description}
		\item[(i)]
		 If $a_1>\lambda_{0,1}d_1$, $a_2\leq\lambda_{0,2}d_2$, then $\lim\limits_{t\to \infty} (u(x,t),v(x,t))=(s_1(x),0)$ provided that $u_0\not\equiv0$, where $s_1$ is the unique positive solution of \eqref{20220904-1} with $i=1$;
	\item[(ii)]
	If $a_1\leq\lambda_{0,1}d_1$, $a_2>\lambda_{0,2}d_2$, then $\lim\limits_{t\to \infty} (u(x,t),v(x,t))=(0,s_2(x))$ provided that $v_0\not\equiv0$, where $s_2$ is the unique positive solution of \eqref{20220904-1} with $i=2$;
		\item[(iii)] 
			Assume that $a_1>\lambda_{0,1} d_1$, $a_2>\lambda_{0,2} d_2$. If 
		\begin{equation}\label{k1}
		a_1-\lambda_{0,1} d_1>\frac{c_1}{c_2} a_2,~ a_2-\lambda_{0,2} d_2> \frac{b_2}{b_1} a_1,
		\end{equation}  $u_0\ge, \not\equiv 0$ and $v_0\ge, \not\equiv 0$, then there exist two positive solutions $\left(\bar{s}(x;\mathcal{M}),\underline{r}(x;\mathcal{M}) \right) $, $\left(\underline{s}(x;\mathcal{M}), \bar{r}(x;\mathcal{M}) \right) $ to \eqref{ss} with ${B}_{i}$ satisfying \eqref{1,5} such that $\left( u,v\right)$ satisfies 
		\begin{equation}\label{458}
		\begin{aligned}
		\underline{s}(x)	\le \liminf\limits_{t\to+\infty} u(x,t)\le \limsup\limits_{t\to +\infty} u(x,t)\le \bar{s}(x),\\
		\underline{r}(x) \le \liminf\limits_{t\to+\infty} v(x,t)\le \limsup\limits_{t\to +\infty} v(x,t)\le \bar{r}(x).\\
		\end{aligned}
		\end{equation}
		 In addition, if 
		 \begin{equation}\label{d}
		 	\tilde{\omega}^{1}\equiv \tilde{\omega}^{2},~\mu^{1}\equiv \mu^{2}~\text{on}~\bar{ \Omega},
		 \end{equation}
	 and
		\begin{equation}\label{tj}
		2 b_{1} \underline{s}(\tilde{\omega}^{1},\mu^{1})>a_{1}-\lambda_{0,1} d_{1}, \quad 2 c_{2} \underline{r}(\tilde{\omega}^{1},\mu^{1})>a_{2}-\lambda_{0,2} d_{2}~(x\in\Omega)~,
		\end{equation}
		then  $\left(\bar{s},\underline{r} \right) \equiv \left(\underline{s}, \bar{r} \right) $ on $\bar{ \Omega}$.
	\end{description} 
\end{Theorem}

\begin{Theorem}\label{1a}
	Let $G=(V,E)$ be a finite graph. Then for any given functions $\tilde{u}_0 \ge 0$, $\tilde{v}_0 \ge 0$ on $V$, \eqref{41} admits a unique solution $(u^{*},v^{*})$ defined for all $t>0$. Furthermore, if $\tilde{u}_0\ge, \not \equiv 0 $ and $\tilde{v}_0\ge, \not \equiv 0$, then $u^{*}(x,t)>0$ and $v^{*}(x,t)>0$ for $x\in V$ and $t>0$.
\end{Theorem}

	\begin{Theorem}\label{m2}
Let $G=G(V,E)$ be a finite graph.	Let $(u,v)$ be the unique solution to \eqref{41} with nonnegative initial value $u_{0},\;v_{0}\not \equiv 0$. Then we have the following conclusions:
		\begin{description}
			\item[(i)]  If $\frac{a_1}{a_2}< \frac{b_1}{b_2}$, $\frac{a_1}{a_2}< \frac{c_1}{c_2}$, then $\lim\limits_{t\to \infty} (u(x,t),v(x,t))=(0,\frac{a_2}{c_2})$;
			\item[(ii)] 	
			If $\frac{b_1}{b_2}<\frac{a_1}{a_2},$ $\frac{c_1}{c_2}<\frac{a_1}{a_2}$, then $\lim\limits_{t\to \infty} (u(x,t),v(x,t))=(\frac{a_1}{b_1},0)$;
			\item[(iii)] 
			If $\frac{c_{1}}{c_{2}}<\frac{a_{1}}{a_{2}}<\frac{b_{1}}{b_{2}}$, then $\lim\limits_{t\to +\infty}\left( u(x,t),v(x,t) \right) =(\xi,\eta)$;
			\item[(iv)] Assume that $\frac{b_{1}}{b_{2}}<\frac{a_{1}}{a_{2}}<\frac{c_{1}}{c_{2}}$. Then $\lim\limits _{t \rightarrow \infty}(u(x, t), v(x, t))=\left(\frac{a_{1}}{b_{1}}, 0\right)$ provided that  $\xi<\tilde{u}_{0}(x)<\frac{a_{1}}{b_{1}}$ and $0<\tilde{v}_{0}(x)<\eta$ on $V$;
			$\lim\limits _{t \rightarrow \infty}(u(x, t), v(x, t))=\left(0,\frac{a_{2}}{c_{2}} \right)$ provided that $0<\tilde{u}_0<\xi$ and $\eta<\tilde{v}_0(x)< \frac{a_2}{c_2}$ on $V$.  
		\end{description} 
	Here, $\xi$ and $\eta$ are defined by \eqref{20220924}.
	\end{Theorem}

We give some common symbols in this paper. $T>0$ is always assumed to be a positive real number.
Denote $\Omega_{T}:=\Omega\times (0,T]$, $\bar{\Omega}_{T}:=\bar{\Omega}\times [0,T]$, $S_{T}:=\partial \Omega\times [0,T]$ and $V_{T}:=V\times(0,T]$.
For an interval $I\subset \mathbb{R}$ and any graphs $G=G(V,E)$, we say that a function $f: V\times I\to \mathbb{R}$ belong to $C^{n}(V\times I)$ if for each $x\in V$, the function $f(x,\cdot)$ is $n-$times differentiable in $I$ and $\frac{\partial ^{n}f}{\partial t ^{n}} (x,\cdot)$ is continuous in $I$, $n=0,1,2.$

The plan of this paper is as follows. In section 2, we establish the maximum principle for weakly coupled parabolic systems on finite graphs. In section 3, we establish the upper and lower solution methods for weakly coupled systems on finite graphs. In section 4, we prove Theorems \ref{m}, \ref{m1} and \ref{m2}. In section 5, we give some numerical results.

\section{Maximum principle for systems on graphs}	

This section is concerned with the maximum principle for functions satisfying a system of parabolic inequalities on finite graphs. 
 	
Let $m>0$ be an integer. We introduce a set of $m$ functions $u_{1}(x,t),~u_{2}(x,t),\cdots,u_{m}(x,t)$, which will be treated as $m$-vector $\mathbf{u}(x,t)$.

\begin{Definition}
For two vectors $\U=(u_1,\cdots, u_m)$ and $\V=(v_1,\cdots, v_m)$,	we say $\U \le \V$ if $u_{i} \le v_{i}$ and $\U<\V$ if $u_{i} < v_{i}$  for $i=1,2,\cdots, m,$ and we write $\U \in <\W,\Z>$ if $w_{i}\le u_{i} \le z_{i}$ for $i=1,\cdots,m.$
\end{Definition}

Let $\bar{ \Omega}$ be a finite graph. Let $\omega^{k}:\bar{ \Omega}\times \bar{ \Omega}\to \mathbb{R}$ be weight functions and $\mu^{k}:\bar{ \Omega}\to \mathbb{R}^{+}$ be finite measure on $\bar{ \Omega}$ for $k=1,\cdots,m.$ 
For any function $u$ on $\bar{ \Omega}$, we define $m$ Laplace operators on $\bar{ \Omega}$ by 
\begin{equation}
	\Delta_{\Omega_{k}} u(x):=\sum_{y \in \bar{\Omega}} (u(y)-u(x))\frac{\omega^{k}_{y x}}{\mu^{k}(x)},~x\in \Omega,~k=1,\cdots,m.
\end{equation}
Given $m$ operators  \begin{equation}\label{pk}
	P_{k}=\partial_{t}-d_{k}\Delta_{\Omega_{k}},
\end{equation} where $d_{k}> 0$ are constants for all $k\in\{1,\cdots,m\}$.

\begin{Lemma}\label{L1}
Let $T>0$.	Assume that $h_{kl}(x,t)$ is bounded on $\Omega_{T}$ and $h_{kl}\le 0$ in $\Omega_{T}$ for $k\not=l$, $k,l=1,\cdots,m$. Suppose that for all $1\le k \le m,$ $u_{k}\in C^{1}(\Omega_{T})\cap C(\bar{\Omega}_{T})$ and that
	\begin{description}
		\item[(I)] $P_{k}u_{k}+\sum\limits_{l=1}^{m} h_{kl}u_{l} <0(>0)$ in $\Omega_{T}$;
		\item[(II)] $u_{k}(x,0)<(>0)$ on $\bar{\Omega}$.
	\end{description}	
If the operators $B_{k}~(k=1,\cdots,m)$ satisfy 
\begin{equation}\label{2.3}
B_{k}u_{k}(x,t)\equiv\frac{\partial u_{k}}{\partial_{\Omega_{k}} n}(x,t):=\sum_{y \in \bar{\Omega}} (u_{k}(x,t)-u_{k}(y,t))\frac{\omega^{k}_{y x}}{\mu^{k}(x)}\le 0(\ge 0)~\text{on}~S_{T},
\end{equation}  then $u_{k}(x,t)<0~ (>0)$ for $x\in \bar{\Omega}$, $t\in [0,T]$ and all $1\le k \le m$;
If the operators $B_{k}~(k=1,\cdots,m)$ satisfy \begin{equation}\label{42}
	B_{k}u_{k}\equiv u_{k}<0~(>0)~\text{on}~S_{T},
\end{equation} then
  $u_{k}(x,t)<0 (>0)$ for $x\in {\Omega}$, $t\in [0,T]$ and all $1\le k \le m$.
\end{Lemma}

\begin{proof}
	Clearly, we can find a constant $\alpha>0$ such that $0<h_{kk}+\alpha$ on $\bar{\Omega}_{T}$ for all $k$. Let $u_{k}=v_{k} e^{\alpha t}$, then by (I), we see that $$P_{k}v_{k}+\sum_{l=1}^{m} \tilde{h}_{kl} v_{l}<0~\text{in}~\Omega_{T}$$ where 
	\begin{equation}\label{x}
		\begin{cases}\tilde{h}_{k k}=h_{k k}+\alpha>0 ~\text{in}~\Omega_{T}, & k=1, \ldots, m, \\ \tilde{h}_{k l}=h_{k l} \leqslant 0~\text{in}~\Omega_{T}, & k \neq l, \quad k, l=1, \ldots, m .\end{cases}
	\end{equation}
Let $\V=(v_1,\cdots,v_{m})$. By (II), there exists $\delta>0$ so that $\V (x,t)<\mathbf{0}$, $x\in \bar{\Omega}$, $t\in [0,\delta]$, where $\mathbf{0}=(0,\cdots,0)$.

 Define $$A:=\{t: 0\le t\le T, \V(x,s)<\mathbf{0}~\text{for}~x\in \bar{\Omega},~s\in [0,t] \}.$$ Let $t_0:=sup A$; then $0<t_0 \le T$. We claim that 
 $$\V(x,t_0)<\mathbf{0} \text{ on }\bar{\Omega}.$$ Otherwise, there exists $x_0 \in \bar{\Omega}$ and some $1\le i \le m$ so that $$v_{i}(x_0, t_0)=0.$$

  If $x_0\in \Omega$, then $v_{i}$ achieves its maximum over $\bar{\Omega}\times[0,t_0]$ at $(x_0,t_0)$, and $\partial_{t} v_{i} (x_0, t_0)\ge 0,$ $\Delta_{\Omega_{i}} v_{i}(x_0, t_0)\le 0$. It follows that ${P}_{i}v_{i} \big| _{(x_0,t_0)}\ge 0$. From \eqref{x}, we deduce that 
\begin{equation}
	P_{i}v_{i}(x_0,t_0)+\sum_{l=1}^{m} \tilde{h}_{il} v_{l} (x_0,t_0)\ge 0.
\end{equation}
This is a contradiction. Thus, we know that $x_0 \in \partial \Omega$ and $v_{i}(x, t_0)<0$ for $x\in \Omega$. 

If $B_{k}~(k=1,\cdots,m)$ satisfy \eqref{42},  then $B_{i}u_{i}(x_0,t_0)=u_{i}(x_0,t_0)\equiv v_{i}(x_0,t_0)e^{\alpha t_0}<0$. This is a contradiction.

We next consider the case where $B_{k}~(k=1,\cdots,m)$ satisfy \eqref{2.3}.
By \eqref{2.3}, we see that
\begin{equation}
	0< \frac{\partial v_{i}}{\partial_{\Omega_{i}} n} \bigg|_{(x_0,t_0)} = \sum_{x\in \bar{\Omega}} [v_{i}(x_0, t_0)-v_{i}(x,t_0)]\frac{\omega^{i}_{x_0 x}}{\mu^{i}(x_0)}\le 0.
\end{equation}
This is a contradiction. 

Therefore, we deduce that $t_0= T$ and hence that $\V <\mathbf{0}$ on $\bar{\Omega}_{T}$. This implies that $\U<0$ on $\bar{\Omega}_{T}$.
\end{proof}
	
	\begin{Theorem}\label{bj}
		Assume that $h_{kl}$ is bounded in $\Omega_{T}$ and $h_{kl}\le 0$ in $\Omega_{T}$ for $k\not=l$, $k,l=1,\cdots,m$. Suppose that for all $k=1,\cdots,m$, \begin{equation}\label{2.4}
			\left\{\begin{array}{lll}
				P_{k} u_{k}+\sum\limits_{l=1}^{m} h_{k l} u_{l} &\leq 0 ~~(\geqslant 0)~~ \text { in } \Omega_{T}, \\
				u_{k}(x, 0) &\leqslant 0~~~~~(\geqslant 0)~~~~~~~~~~ \text { on } \bar{\Omega}, \\
				B_{k} u_{k} &\leqslant 0~~~~~(\geqslant 0)~~  \text { on } S_{T}.
			\end{array}\right.
		\end{equation}
Here, $B_{k}~(k=1,\cdots,m)$ satisfy \begin{equation}\label{2,3}
	B_{k}u_{k}(x,t)\equiv\frac{\partial u_{k}}{\partial_{\Omega_{k}} n}(x,t):=\sum_{y \in \bar{\Omega}} (u_{k}(x,t)-u_{k}(y,t))\frac{\omega^{k}_{y x}}{\mu^{k}(x)},
\end{equation} 
or \begin{equation} \label{2,4}
		B_{k}u_{k}(x,t)\equiv u_{k}(x,t) ~\text{on}~S_{T}.
\end{equation}
	Then we have $\U \le 0$ $(\ge 0)$ in $\bar{\Omega}_{T}$. Furthermore, suppose that some component $u_{i}$ of $\U$ satisfies  $u_{i}(x,0)\not\equiv 0$ on $\bar{\Omega}$. Then if $B_{k}~(k=1,\cdots,m)$ satisfy \eqref{2,3}, then $u_{i}(x,t)<0$ on $\bar{\Omega}\times (0,T]$ for all $ k=1,\cdots,m$; If $B_{k}~(k=1,\cdots,m)$ satisfy \eqref{2,4}, then $u_{i}(x,t)<0$ on ${\Omega}\times (0,T]$ for all $ k=1,\cdots,m$.
	\end{Theorem}
\begin{proof}
	Thanks to $h_{kl}$ is bounded, there is $\beta>0$ such that $\beta+\sum\limits_{i=1}^{m} h_{ki}>0$ in $\Omega_{T}$ for $ 1\le k \le m$. For any given $\epsilon>0$, let $v_{k}=u_{k}-\epsilon e^{\beta t}$, $k=1,\cdots,m$; then 
	\begin{equation}
		P_{k}v_{k}+\sum_{i=1}^{m}h_{ki}v_{i}=P_{k}u_{k}+\sum_{i=1}^{m} h_{ki} u_{i}-\epsilon e^{\beta t}\left( \beta+\sum_{i=1}^{m} h_{ki} \right)<0, 
	\end{equation}
\begin{equation}
	v_{k}(x,0)<0~\text{ on }~\bar{\Omega}, 
\end{equation}

If $B_{k}~(k=1,\cdots,m)$ satisfy \eqref{2,4}, then
\begin{equation}
	 v_{k}\equiv  u_{k}-\varepsilon e^{\beta t} < 0 \text { on } S_{T}~(k=1,\cdots,m).
\end{equation}
By Lemma \ref{L1}, $v_{k}<0$ in ${\Omega}_{T}$ for all $ k=1,\cdots,m$. Letting $\epsilon\to 0$; then we obtain $u_{k}\le 0$ on ${\Omega}_{T}$ for all $ k=1,\cdots,m$. By \eqref{2.4}, we know that $u_{k}\le 0$ on $\bar{\Omega}_{T}$ for all $ k=1,\cdots,m$.

If $B_{k}~(k=1,\cdots,m)$ satisfy \eqref{2,3}, then \begin{equation}
	\frac{\partial v_{k}}{\partial_{\Omega_{k}} n}\equiv	B_{k}u_{k}\le 0\text { on } S_{T}~(k=1,\cdots,m).
\end{equation}
By Lemma \ref{L1}, $v_{k}<0$ in $\bar{\Omega}_{T}$ for $ 1\le k \le m$. Letting $\epsilon\to 0$; then we obtain $u_{k}\le 0$ on $\bar{\Omega}_{T}$ for all $ k=1,\cdots,m$.

By the first inequality in \eqref{2.4}, we see that 
\begin{equation}
	P_{i}u_{i}+h_{ii} u_{i}\le P_{i}u_{i}+h_{ii} u_{i} +\sum_{j:j\not=i}h_{ij}u_{j}\le 0.
\end{equation}
By Theorem 3.5 in \cite{HL}, we deduce that if $B_{k}~(k=1,\cdots,m)$ satisfy \eqref{2,3}, then $u_{i}(x,t)<0$ on $\bar{\Omega}\times (0,T]$ for all $ k=1,\cdots,m$; If $B_{k}~(k=1,\cdots,m)$ satisfy \eqref{2,4}, then $u_{i}(x,t)<0$ on ${\Omega}\times (0,T]$ for all $ k=1,\cdots,m$.
\end{proof}

Let $G=(V,E)$ be a finite graph. Let $\tilde{\omega}^{k}:V \times V \to \mathbb{R}$ be weight functions and $\tilde{\mu}^{k}:V \to \mathbb{R}^{+}$ be finite measure on $V$ for $k=1,\cdots,m.$ 
For any function $u$ on $V$, we define $m$ Laplace operators on $V$ by 
\begin{equation}
	\Delta_{V_{k}} u(x):=\sum_{y \in V} (u(y)-u(x))\frac{\tilde{\omega}^{k}_{y x}}{\tilde{\mu}^{k}(x)},~x\in V,~k=1,\cdots,m.
\end{equation}
Let $\tilde{d}_{k}> 0$ be constants, $k\in\{1,\cdots,m\}$. Define operators 
\begin{equation}\label{pk2}
	\tilde{P}_{k}:=\partial_{t}-\tilde{d}_{k}\Delta_{V_{k}},
\end{equation}
  where $k\in\{1,\cdots,m\}$.
  
  By similar arguments as in the proof of Lemma \ref{L1}, we have the following result.
\begin{Lemma}\label{23}
 Suppose that $h_{kl}(x,t)$ is bounded on $V_{T}$ and $h_{kl}\le 0$ in $V_{T}$ for $k\not=l$, $k,l=1,\cdots,m$ and $u_{k}\in C^{1}(V_{T})$. Suppose that for all $k=1,\cdots,m,$
	\begin{description}
		\item[(I)] $\tilde{P}_{k}u_{k}+\sum\limits_{l=1}^{m} h_{kl}u_{l} <0(>0)$ on $V_{T}$;
		\item[(II)] $u_{k}(x,0)<(>0)$ on $V$.
	\end{description}	
Then $u_{k}(x,t)<0 (>0)$ for $x\in V$, $t\in [0,T]$ and all $k=1,\cdots,m$.
\end{Lemma}
\begin{proof}
	
Since $h_{kl}$ is bounded on ${V_{T}}$, there exists $\alpha>0$ so that $0<h_{kk}+\alpha$ on $V_{T}$ for all $k$. Let $\V=(v_1,\cdots,v_{m})$ satisfying $u_{k}=v_{k}e^{\alpha t}$; then by (I), we see that $\tilde{P}_{k}v_{k}+\sum\limits_{l=1}^{m} \tilde{h}_{kl} v_{l}<0$, where 
\begin{equation}\label{x1}
	\begin{cases}\tilde{h}_{k k}=h_{k k}+\alpha>0, & k=1, \ldots, m, \\ \tilde{h}_{k l}=h_{k l} \leqslant 0, & k \neq l, \quad k, l=1, \ldots, m .\end{cases}
\end{equation}
 By (II), there exists $\delta>0$ so that $\V (x,t)<0$, $x\in V$, $t\in [0,\delta]$.

 Define $A:=\{t: t\le T, \V(x,s)<0~\text{for}~x\in V,~s\in [0,t] \}$. Let $t_0:=sup A$; then $0<t_0 \le T$. We claim that $\V(x,t_0)<0$ on $V$. Otherwise, we can find $x_0 \in V$ and some $i\in\{1,\cdots,m\}$ so that $v_{i}(x_0, t_0)=0$. By \eqref{x1}, this implies that 	\begin{equation}
 	\tilde{P}_{i}v_{i}(x_0,t_0)+\sum_{l=1}^{m} \tilde{h}_{il} v_{l} (x_0,t_0)\ge 0.
 \end{equation} This is a contradiction. Thus, $t_0=T$ and $\V<0$ on $V\times [0,T]$.
\end{proof}

By similar discussions as in the proof of Theorem \ref{bj}, we have the following result.
	\begin{Theorem}\label{1.4}
	Assume that $h_{kl}$ is bounded and $h_{kl}\le 0$ on $V_{T}$ for $k\not=l$, $k,l=1,\cdots,m$. Suppose that for all $1\le k \le m$, \begin{equation}\label{24}
		\left\{\begin{array}{l}
			\tilde{P}_{k} u_{k}+\sum\limits_{l=1}^{m} h_{k l} u_{l} \leq 0 \quad(\geqslant 0) \text { on } V_{T}, \\
			u_{k}(x, 0) \leqslant 0(\geqslant 0) \text { in } V.
		\end{array}\right.
	\end{equation}
	Then we have $\U \le 0$ ($\ge 0$) on $V_{T}$. Furthermore, if $u_{k}(x,0)\not\equiv 0$ on $V$ for some $k\in\{1,\cdots,m\}$, then $u_{k}(x,t)<0~(>0)$ on $V_{T}$.
\end{Theorem}
\begin{proof}
	Thanks to $h_{kl}$ is bounded, there is $\beta>0$ such that $\beta+\sum\limits_{i=1}^{m} h_{ki}>0$ in $V_{T}$ for $k=1,\cdots,m$. Let $v_{k}=u_{k}-\epsilon e^{\beta t}$ with $\epsilon>0$. Then 
	\begin{equation}
		\tilde{P}_{k}v_{k}+\sum_{i=1}^{m}h_{ki}v_{i}=\tilde{P}_{k}u_{k}+\sum_{i=1}^{m} h_{ki} u_{i}-\epsilon e^{\beta t}\left( \beta+\sum_{i=1}^{m} h_{ki} \right)<0, 
	\end{equation}
	\begin{equation}
		v_{k}(x,0)<0~\text{ on }~V.
	\end{equation}
	By Lemma \ref{23}, $v_{k}<0$ on $V_{T}$, $k=1,\cdots,m$. Letting $\epsilon\to 0$. Then we obtain $u_{k}\le 0$ on $V_{T}$, $k=1,\cdots,m$.
	
	Suppose $u_{k}(x,0)\not \equiv 0$ on $V$. By the first inequality in \eqref{24}, we see that 
	\begin{equation}
		\tilde{P}_{k}u_{k}+h_{kk} u_{k}\le \tilde{P}_{k}u_{k}+h_{kk} u_{k} +\sum_{i\not=k}h_{ki}u_{i}\le 0.
	\end{equation}
	By Theorem 3.7 in \cite{HL}, we deduce that $u_{k}(x,t)<0$ on $V_{T}$.
\end{proof}

\section{Monotone methods for systems on graphs} 
	
In this section, we establish upper and lower solutions method for weak coupled parabolic systems on finite connected graphs.	
	
 We write the vector $\U$ in the split form $\U=(\U_{i}, [\U]_{i_{k}},[\U]_{d_{k}})$ and rewrite the functions $f_{k}$ as $f_{k}(t,x,\U)=f_{k}(t,x,u_{k},[\U]_{i_{k}}, [\U]_{d_{k}})$ $k=1\cdots,m$, where $i_{k}$, $d_{k}$ denote the $a_{i}-$components and $b_{i}-$components of the vector $\U$, respectively. The split form of $u$ varies with respect to $i$ and is determined by the quasimonotone property of $f_{k}.$ 

\begin{Definition}
	A vector function $\F=(f_1,\cdots,f_{m})$ is said to be mixed qusimonotonous in the order interval $<\W,\Z>$ if there exist nonnegative integers $i_{k}$, $d_{k}$ with $i_{k}+d_{k}=m-1$ such that for each $1\le k\le m$, and all $u\in <\W,\Z>$, $f_{k}(\cdot,u_{k},[\U]_{i_{k}}, [\U]_{d_{k}})$ is monotonically nondecreasing in $[\U]_{i_{k}}$, i.e., $f_{k}$ is nondecreasing with respect to each component of $[\U]_{i_{k}}$, and monotonically nonincreasing in $[\U]_{d_{k}}$, i.e., $f_{k}$ is nonincreasing with respect to each component of $[\U]_{d_{k}}$. If $i_{k}=0$ or $d_{k}=0$ for all $k$, we say $f$ is qusimonotone nonincreasing and qusimonotone nondecreasing, respectively.
\end{Definition} 
  	
 Define vector functions
\begin{equation}
	\F(x,t,\U)=\left(f_{1}(x,t,\U), \ldots, f_{m}(x,t,\U) \right), \quad \G(x,t)=\left(g_{1}(x,t), \ldots, g_{m}(x,t) \right).
\end{equation}  	
  Consider the following problem
  \begin{equation}\label{1}
  	\begin{cases}P_{k} u_{k}=f_{k}(x, t, \U) & \text { in } \Omega_{T}, \\ B_{k} u_{k}=g_{k} & \text { on } S_{T}, \\ u_{k}(x, 0)=\varphi_{k}(x) & \text { on } \bar{\Omega}, \\ k=1, \ldots, m, & \end{cases}
  \end{equation}	
  	where $\bar{ \Omega}$ is a finite graph, $P_{k}~(k=1,\cdots,m)$ are defined by \eqref{pk}, and $B_{k}~(k=1,\cdots,m)$ satisfy \eqref{2,3} or \eqref{2,4}.
  
  \begin{Definition}\label{32}
  	For $\overline{\U}=(\overline{u}_{1},\cdots,\overline{u}_{m})$ and $\underline{\U}=(\underline{u}_{1},\cdots,\underline{u}_{m})$ .
  	Define $\underline{\psi}_{k}:=\min
  	\{ \bar{u}_{k}, \underline{u}_{k} \}$,  $\overline{\psi}_{k}:=\max
  	\{ \bar{u}_{k}, \underline{u}_{k} \}$ for $1\le k \le m$, $\underline{\boldsymbol{\psi}}=(\underline{\psi}_{1},\cdots,\underline{\psi}_{m}),$ and $\bar{\boldsymbol{\psi}}=(\overline{\psi}_{1},\cdots,\overline{\psi}_{m})$. 
  \end{Definition}	
  	
 \begin{Definition}\label{3}
 	 Assume that $\F$ is mixed quasi-monotone in  $<\underline{{\boldsymbol{\psi}}},\overline{{\boldsymbol{\psi}}}>$. We say that $(\bar{\U},\underline{\U})$ is a pair of coupled upper and lower solutions of \eqref{1} if the following holds:
 	\begin{equation}
 		\begin{cases} P_{k} \bar{u}_{k} \geqslant f_{k}\left(x, t, \bar{u}_{k},[\bar{\U}]_{\mathrm{i}_{k}},[\underline{\U}]_{\mathrm{d}_{k}}\right) & \text { in } \Omega_{T}, \\ P_{k} \underline{u}_{k} \leqslant f_{k}\left(x, t, \underline{u}_{k},[\underline{\U}]_{\mathrm{i}_{k}},[\bar{\U}]_{\mathrm{d}_{k}}\right) & \text { in } Q_{T}, \\ {B}_{k} \bar{u}_{k} \geqslant g_{k} \geqslant {B}_{k} \underline{u}_{k} & \text { on } S_{T}, \\ \bar{u}_{k}(x, 0) \geqslant \varphi_{k}(x) \geqslant \underline{u}_{k}(x, 0) & \text { on } \bar{\Omega}, 
 		\end{cases}
 	\end{equation} 	
 for $1 \leqslant k \leqslant m$, where $\underline{\boldsymbol{\psi}}$ and $\bar{\boldsymbol{\psi}}$ are defined as in Definition \ref{32}.
 \end{Definition} 	
 
 Similarily, we can define the coupled upper and lower solutions of the following problem \begin{equation}\label{14}
 	\begin{cases}\tilde{P}_{k} u_{k}=f_{k}(x, t, \U) & \text { in } V_{T}, \\ u_{k}(x, 0)=g_{k}(x) & \text { on } V, \\ k=1, \ldots, m, & \end{cases}
 \end{equation}
where $G=(V,E)$ is a finite graph and $\tilde{P}_{k}$ $(k=1,\cdots,m)$ are defined by \eqref{pk2}.
 \begin{Definition}\label{34}
 	 Assume that $\F$ is mixed qusi-monotone in $<\underline{\boldsymbol{\psi}},\bar{\boldsymbol{\psi}}>$. We say that $(\bar{\U},\underline{\U})$ is a pair of coupled upper and lower solutions of \eqref{14}, if $(\bar{\U},\underline{\U})$ satisfies:
 	\begin{equation}\label{315}
 		\begin{cases} \tilde{P}_{k} \bar{u}_{k} \geqslant f_{k}\left(x, t, \bar{u}_{k},[\bar{\U}]_{\mathrm{i}_{k}},[\underline{\U}]_{\mathrm{d}_{k}}\right) & \text { in } V_{T}, \\ \tilde{P}_{k} \underline{u}_{k} \leqslant f_{k}\left(x, t, \underline{u}_{k},[\underline{\U}]_{\mathrm{i}_{k}},[\bar{\U}]_{\mathrm{d}_{k}}\right) & \text { in } V_{T}, \\  \bar{u}_{k}(x, 0) \geqslant \varphi_{k}(x) \geqslant \underline{u}_{k}(x, 0) & \text { on } V, 
 		 \end{cases}
 	\end{equation} 	
 for $1 \leqslant k \leqslant m$, where $\underline{\boldsymbol{\psi}}$ and $\bar{\boldsymbol{\psi}}$ are defined as in Definition \ref{32}.
 \end{Definition}
 
 Next, we establish the comparison principle for the problem \eqref{1}.
 \begin{Theorem}\label{21}
  Assume that $(\bar{\U},\underline{\U})$ is a pair of coupled upper and lower solutions of \eqref{1}. If $\F$ satisfies the Lipschitz condition in $<\underline{\boldsymbol{\psi}},\bar{\boldsymbol{\psi}}>,$ i.e., there exists $M>0$ so that 
 	\begin{equation}\label{s462}
 		\left| f_{k}(x,t,\U)-f_{k}(x,t,\V) \right| \le M\left| \U-\V \right|  
 	\end{equation}
 	for $(x,t)\in \Omega_{T}$, $\U, \V\in <\underline{{\boldsymbol{\psi}}},\overline{{\boldsymbol{\psi}}}>$ and $1\le k \le m$,  where $\underline{\boldsymbol{\psi}}$ and $\bar{\boldsymbol{\psi}}$ are defined as in Definition \ref{32}. Then $\bar{\U}\ge \underline{\U}$ in $\Omega_{T}$.
 \end{Theorem} 	
  	
\begin{proof}
For every $k\in\{1,\cdots,m\}$, we suppose that 
where integers $j_1,j_2,\cdots,j_{i_{k}}$, $l_1,l_2,\cdots,l_{d_{k}}$ satisfying 
	$$\{ k,j_1,j_2,\cdots,j_{i_{k}},l_1,l_2,\cdots,l_{d_{k}}\}=\{1,2,\cdots,m\}$$ and 
	$j_1<j_2<\cdots<j_{i_k}, l_1<l_2<\cdots<l_{d_k}$.
	Let $w_{k}:=\overline{u}_{k}-\underline{u}_{k}$; then $w_{k}$ satisfies 
	\begin{equation}
		\begin{aligned}
			&P_{k} w_{k} \geqslant f_{k} \left(x, t, \bar{u}_{k},[\bar{\U}]_{i_ k}, [\underline{\U}]_{d_ k}\right)-f_{k}\left(x, t, \underline{u}_{k},[\underline{\U}]_{i_ {k}},[\bar{\U}]_{d_{k}}\right)\\		
			&=c_{k}(x, t)\left(\bar{u}_{k}-\underline{u}_{k}\right)+\sum_{j:u_{j} \in\left[\U\right]_{i_ k}} c_{k, j}(x, t) w_{j}-\sum_{j:u_{j} \in\left[\U\right]_{d _k}} D_{k, j}(x, t) w_{j},
		\end{aligned}
	\end{equation}
where 
\begin{equation*}
c_{k}(x,t)	=\left\{\begin{array}{cl}
		\frac{f_{k}\left(x, t, \bar{u}_{k},[\bar{\U}]_{i_ k},[\underline{\U}]_{d_{k}} \right)-f_{k}\left(x, t, \underline{u}_{k},[\bar{\U}]_{i_ k},[\underline{\U}]_{d_ k}\right)}{\bar{u}_{k}-\underline{u}_{k}} & ,~ \text{if}~ \bar{u}_{k} \neq \underline{u}_{k}, \\
		0, & ,~ \text{if}~ \bar{u}_{k}=\underline{u}_{k} ,
	\end{array}\right.
\end{equation*}
\begin{equation*}
	c_{k,j_1}(x,t)=\left\{ \begin{array}{cl}
		\frac{f_{k}\left(\cdot, \cdot, \underline{u}_{k},[\bar{\U}]_{i_ k},[\underline{\U}]_{d_ k} \right)-f_{k}\left(\cdot, \cdot, \underline{u}_{k},\left(\underline{u}_{j_1}, \bar{u}_{j_2}, \cdots, \bar{u}_{j_{i_ k}}\right),[\underline{\U}]_{d_ k}\right)}{\bar{u}_{j_1}-\underline{u}_{j_1}}&, ~ \text{if}~\bar{u}_{j_1} \neq \underline{u}_{j_1}, \\
		0&, ~ \text{if}~\bar{u}_{j_1}=\underline{u}_{j_1},
	\end{array}\right.
\end{equation*}
\begin{equation*}
	c_{k,j_2}(x,t)=\left\{ \begin{array}{cl}
		\frac{f_{k}\left(\cdot, \cdot , \underline{u}_{k},\left(\underline{u}_{j_1}, \bar{u}_{j_2}, \cdots, \bar{u}_{j_{i_ k}}\right),[\underline{\U}]_{d_ k}\right)-f_{k}\left(\cdot,\cdot, \underline{u}_{k,}\left(\underline{u}_{j_1}, \underline{u}_{j_2}, \bar{u}_{j_3}, \cdots, \bar{u}_{j_{i_ k} }\right), [\underline{\U}]_{d_ k}\right)}{\bar{u}_{j_2}-\underline{u}_{j_2}}&,~ \text{if}~\bar{u}_{j_2} \not=\underline{u}_{j_2}, \\
		0&, ~ \text{if}~\bar{u}_{j_2}=\underline{u}_{j_2},
	\end{array}\right.
\end{equation*}
$\cdots$
\begin{equation*}
	c_{k,j_{i_k}}(x,t)=\left\{ \begin{array}{cl}
		\frac{f_{k}\left(\cdot, \cdot , \underline{u}_{k},\left(\underline{u}_{j_1}, \cdots, \underline{u}_{j_ {k}-1}, \bar{u}_{j_{i_k}}\right),[\underline{\U}]_{d_ k}\right)-f_{k}\left(\cdot,\cdot, \underline{u}_{k}, [\underline{\U}]_{i_{k}}, [\underline{\U}]_{d_ k} \right)}{\bar{u}_{j_{i_k}}-\underline{u}_{j_{i_k}}}&,~ \text{if}~\bar{u}_{j_{i_k}} \not=\underline{u}_{j_{i_k}}, \\
		0&, ~ \text{if}~\bar{u}_{j_{i_k}}=\underline{u}_{j_{i_k}},
	\end{array}\right.
\end{equation*}
and
\begin{equation*}
	D_{k,l_1} (x,t)= \left\{ \begin{array}{cl}	
	\frac{f_{k}\left(\cdot, \cdot,\underline{u}_{k},[\underline{\U}]_{i_ {k}},[\underline{\U}]_{d_ k}\right)-f_{k}\left(\cdot,\cdot, \underline{u}_{k},[\underline{\U}]_{i_{k}},\left(\bar{u}_{l_1 }, \underline{u}_{l_2}, \cdots, \underline{u}_{l_{d_{k}} }\right)\right)}{\left(\underline{u}_{l_1}-\bar{u}_{l_1}\right)}
				  &, ~ \text{if}~\underline{u}_{l_1} \not=\bar{u}_{l_1}, \\
				0 &,~ \text{if}~ \underline{u}_{l_1}=\bar{u}_{l_1} ,				
			\end{array}\right.
\end{equation*}
\begin{equation*}
\begin{aligned}
	&	D_{k,l_2} (x,t)= \\	
	~ &\left\{ \begin{array}{cl}	
	\frac{f_{k}\left(\cdot,\cdot, \underline{u}_{k},[\underline{\U}]_{i_ {k}},(\bar{u}_{l_1},\underline{u}_{l_2},\cdots,\underline{u}_{l_{d_{k}}}     )  \right)-f_{k}\left(\cdot,\cdot, {u}_{k},[\underline{\U}]_{i_{k}},\left( \bar{u}_{l_1 }, \overline{u}_{l_2}, \underline{u}_{l_3} \cdots, \underline{u}_{l_{d_{k}} }\right)\right)}  {{\underline{u}}_{l_2}-\bar{u}_{l_2}}
		, &\text{if}~\underline{u}_{l_2} \not=\bar{u}_{l_2}, \\
		 0, &\text{if}~  \underline{u}_{l_2}=\bar{u}_{l_2} ,				
	\end{array}\right.
\end{aligned}
\end{equation*}
$\cdots$
\begin{equation*}
	\begin{aligned}
		&	D_{k,l_{d_{k}}} (x,t)= \\	
		~ &\left\{ \begin{array}{cl}	
			\frac{f_{k}\left(\cdot,\cdot, {u}_{k},[\underline{\U}]_{i_ {k}},(\bar{u}_{l_1},\cdots,\bar{u}_{l_{d_{k-1}}},\underline{u}_{l_{d_{k}}}     )  \right)-f_{k}\left(\cdot,\cdot, {u}_{k},[\underline{\U}]_{i_{k}},\left( \bar{u}_{l_{1} }, \cdots,\overline{u}_{l_{d_{k}-1}},   \bar{u}_{l_{d_{k}} }\right)\right)}  {{\underline{u}}_{l_{d_{k}}}-\bar{u}_{l_{d_{k}}}}
			,  &\text{if}~\underline{u}_{l_{d_{k}}} \not=\bar{u}_{l_{d_{k}}}, \\
			0 ,& \text{if}~ \underline{u}_{l_{d_{k}}}=\bar{u}_{l_{d_{k}}} .				
		\end{array}\right.
	\end{aligned}
\end{equation*}
Let 
\begin{equation}\label{38}
	a_{kk}=-c_{k}, a_{kj}=-c_{k,j}~\text{if}~j:~u_{j}\in [\U]_{i_{k}},~ \text{and}~ a_{kj}=D_{k,j}~ \text{if}~j: u_{j}\in [\U]_{d_k}.
\end{equation}
Then we see that $\{a_{ij}\}_{i,j=1}^{m}$ is bounded in $\Omega_{T}$, $a_{kj}\le 0$ if $k\not=j$, and $P_{k}w_{k}+\sum\limits_{j=1}^{m} a_{k{j}} w_{j} \ge 0$ in $\Omega_{T}$. Clearly, $B_{k}w_{k}\ge 0$ on $S_{T}$ and $w_{k}(x,0)\ge 0$ on $\Omega$. By Theorem \ref{bj}, we have $w_{k}\ge 0$ in $\Omega_{T}$ for $1\le k \le m$, i.e., $\bar{u}_{k}\ge \underline{u}_{k}$ in $\Omega_{T}$ for $1\le k \le m$.
\end{proof} 	
 
 In the following, we establish the upper and lower solution method for the problem \eqref{1}.
 \begin{Theorem}\label{sx}
 	Let $(\bar{\U},\underline{\U})$ be a pair of coupled upper and lower solutions of \eqref{1} and $\F$ satisfy the condition of Theorem \ref{21}. Assume that $\F$ has mixed qusimonotonicity property in $<\underline{\U},\bar{\U}>$, and $\F$ satisfies the Lipschitz condition in $<\underline{\U},\bar{\U}>$, i.e., f satisfies \eqref{s462} for $(x,t)\in \Omega_{T}$,  $\U, \V\in <\underline{\U},\bar{\U}>$ and $1\le k \le m$. Then there exist two monotone sequences $\{\underline{\U}^{i}\}_{i=1}^{\infty}$ and $\{\overline{\U}^{i}\}_{i=1}^{\infty}$ satisfying 
 	\begin{equation}
 		\underline{\U} \leqslant \underline{\U}^{i} \leqslant \underline{\U}^{i+1} \leqslant \bar{\U}^{i+1} \leqslant \bar{\U}^{i} \leqslant \bar{\U} \text { for all } i \geqslant 1, \lim\limits_{i\to \infty} \underline{\U}^{i}=\lim\limits_{i\to \infty} \overline{\U}^{i}=\U
 	\end{equation}
 and $\U$ is the unique solution to \eqref{1} in the order interval $<\underline{\U},\overline{\U}>$.
 \end{Theorem} 	
  \begin{proof}
  	By Theorem \ref{21}, we have $\underline{\U}\le \bar{\U}$. For any given $\V,\W\in <\underline{\U},\overline{\U}>$, it is well-known that the problem 
  	\begin{equation}
  		\begin{cases}P_{k} u_{k}+M u_{k}=f_{k}\left(x, t, v_{k},[\V]_{\mathrm{i}_{k}},[\mathbf{w}]_{\mathrm{d}_{k}}\right)+M v_{k} & \text { in } \Omega_{T}, \\ {B}_{k} u_{k}=g_{k} & \text { on } S_{T}, \\ u_{k}(x, 0)=\varphi_{k}(x) & \text { on } \bar{ \Omega} \end{cases}
  	\end{equation}
  admits a unique solution $u_{k}$, $1\le k \le m$. Thus we can define an operator $\mathcal{F}$ by $\U=\mathcal{F}(\V,\mathbf{w})$ and 
  \begin{equation}
  	\underline{\U}^{1}=\mathcal{F}(\underline{\U}, \bar{\U}), \quad \bar{\U}^{1}=\mathcal{F}(\bar{\U}, \underline{\U}), \quad \underline{\U}^{i+1}=\mathcal{F}\left(\underline{\U}^{i}, \bar{\U}^{i}\right),~ \text { and } \bar{\U}^{i+1}=\mathcal{F}\left(\bar{\U}^{i}, \underline{\U}^{i}\right).
  \end{equation}

 In the following, we always assume that $k$ is any number from $1$ to $m$.

It is easy to check that $v_{k}:=\overline{u}_{k}-\overline{u}_{k}^{1}$ satisfies 
\begin{equation}
	\begin{cases} P_{k} v_{k}+M v_{k} \geqslant 0 & \text { on } \Omega_{T}, \\ {B}_{k} v_{k} \geqslant 0 & \text { on } S_{T}, \\ v_{k}(x, 0) \geqslant 0 & \text { on } \bar{\Omega} .\end{cases}
\end{equation}
By Theorem \ref{bj}, we have $v_{k}\ge 0$, i.e., $\bar{u}_{k}\ge \bar{u}_{k}^{1}$. By a similar discussion, we have $\underline{u}_{k}\le \underline{u}_{k}^{1}$. Set $w_{k}:=\bar{u}_{k}^{1}-\underline{u}_{k}^{1}$. By virtue of $\bar{\U}\ge \underline{\U}$, by the mixed qusimonotonicity of $f_{k}$, we know that 
\begin{equation}
	\begin{aligned}
		P_{k} w_{k}+M w_{k} &=f_{k}\left(\cdot,\cdot, \bar{u}_{k},[\bar{\U}]_{\mathrm{i}_{k}},[\underline{\U}]_{\mathrm{d}_{k}}\right)-f_{k}\left(\cdot,\cdot, \underline{u}_{k},[\underline{\U}]_{\mathrm{i}_{k}},[\bar{\U}]_{\mathrm{d}_{k}}\right)+M\left(\bar{u}_{k}-\underline{u}_{k}\right) \\
		& \geqslant f_{k}\left(\cdot,\cdot, \bar{u}_{k},[\underline{\U}]_{\mathrm{i}_{k}},[\underline{\U}]_{\mathrm{d}_{k}}\right)-f_{k}\left(\cdot,\cdot, \underline{u}_{k},[\underline{\U}]_{\mathrm{i}_{k}},[\underline{\U}]_{\mathrm{d}_{k}}\right)+M\left(\bar{u}_{k}-\underline{u}_{k}\right) \\
		& \geqslant 0 \quad \text { in } \Omega_{T}, \\
		{B}_{k} w_{k} &\ge0 \quad \text { on } S_{T}, \\
		w_{k}(x, 0) &\ge0 \quad \text { in } \Omega .
	\end{aligned}
\end{equation}
By Theorem \ref{bj}, $w_{k}\ge 0$, i.e., $\bar{u}_{k}^{1} \ge \underline{u}_{k}^{1}$. Let $z_{k}:=\bar{u}_{k}^{1}-\bar{u}_{k}^{2};$ then we obtain 
\begin{equation}
	\begin{aligned}
		P_{k} z_{k}+M z_{k} &=f_{k}\left(\cdot,\cdot, \bar{u}_{k},\left[\bar{\U}\right] _{\mathrm{i}_{k}},[\underline{\U}]_{\mathrm{d}_{k}}\right)-f_{k}\left(\cdot,\cdot, \bar{u}_{k}^{1},\left[\bar{\U}^{1}\right]_{\mathrm{i}_{k}},\left[\underline{\U}^{1}\right]_{\mathrm{d}_{k}}\right)+M\left(\bar{u}_{k}-\bar{u}_{k}^{1}\right)\\
		& \geqslant f_{k}\left(\cdot,\cdot, \bar{u}_{k},\left[\bar{\U}^{1}\right]_{\mathrm{i}_{k}},[\underline{\U}]_{\mathrm{d}_{k}}\right)-f_{k}\left(\cdot,\cdot, \bar{u}_{k}^{1},\left[\bar{\U}^{1}\right]_{\mathrm{i}_{k}},[\underline{\U}]_{\mathrm{d}_{k}}\right)+M\left(\bar{u}_{k}-\bar{u}_{k}^{1}\right) \\
		& \geqslant 0 \quad \text { in } \Omega_{T}, \\
	{B}_{k} z_{k} &\ge0 \quad \text { on } S_{T}, \\
		z_{k}(x, 0) &\ge0 \quad \text { on } \bar{\Omega} .
	\end{aligned}
\end{equation}
By Theorem \ref{bj}, $z_{k}\ge 0$, i.e., $\bar{u}_{k}^{1}\ge \bar{u}_{k}^{2}$. By a similar argument as above, we have $\underline{u}_{k}^{1}\le \underline{u}_{k}^{2}$.

By induction, we have $\underline{\U}\le \underline{\U}^{i}\le \underline{\U}^{i+1}\le \bar{\U}^{i+1}\le \bar{\U}^{i} \le \bar{\U}$ for $i\in\mathbb{Z}^+$. Thus $\hat{\U}:=\lim\limits_{i\to \infty}\bar{\U}^{i}$ and $\tilde{\U}:=\lim\limits_{i\to \infty}\underline{\U}^{i}$ are well-defined. Clearly, $\hat{\U}\ge \tilde{\U}$. Since 
\begin{equation*}
	P_{k} \bar{u}_{k}^{i+1} + M\bar{u}_{k}^{i+1} = f_{k} (x,t,\bar{u}_{k}^{i},[\bar{\U}^{i}]_{i_{k}}, [\underline{\U}^{i}]_{d_{k}} ) +M\bar{u}_{k}^{i}~\text{ in }~\Omega_{T}, 
\end{equation*}
we see that for any $t_1,~t_2$ satisfying $(t_1, t_2)\subset (0,T]$, 
\begin{equation}\label{214}
	\bar{u}_{k}^{i+1}(x,t_2) -\bar{u}_{k}^{i+1}(x,t_1)=\int_{t_1}^{t_2} \left[d_{k}\Delta_{\Omega_{k}} \bar{u}^{i+1}_{k} -M \bar{u}^{i+1}_{k} +f_{k}(x,t,\bar{u}^{i}_{k},[\bar{\U}^{i}]_{i_{k}}, [\underline{\U}^{i}]_{d_{k}}) + M \bar{u}^{i}_{k} \right]~dt.
\end{equation}
Letting $i\to \infty$ in \eqref{214}; then we have 
\begin{equation}\label{}
	\hat{u}_{k}(x,t_2) -\hat{u}_{k}(x,t_1)=\int_{t_1}^{t_2}\left[ d_{k}\Delta_{\Omega_{k}} \hat{u}_{k}-M\hat{u}_{k} +f_{k}(x,t,\hat{u}_{k},[\hat{\U}]_{i_{k}}, [\tilde{\U}]_{d_{k}}) + M \hat{u}_{k} \right]~dt.
\end{equation}
This implies that 
\begin{equation}
	\frac{\partial \hat{u}_{k}(x,t)}{\partial t}=d_{k}\Delta_{\Omega_{k}} \hat{u}_{k} (x,t) + f_{k} (x,t,\hat{u}_{k},[\hat{\U}]_{i_{k}},[\tilde{\U}]_{d_{k}})~\text{ in }~\Omega_{T}.
\end{equation}
In the same way, 
\begin{equation}
	\frac{\partial \tilde{u}_{k}(x,t)}{\partial t}=d_{k}\Delta_{\Omega_{k}} \tilde{u}_{k} (x,t) + f_{k} (x,t,\tilde{u}_{k},[\tilde{\U}]_{i_{k}},[\hat{\U}]_{d_{k}})~\text{ in }~\Omega_{T}.
\end{equation}
We next show that $\hat{u}_{k}\equiv\tilde{u}_{k}$ for all $k$. In view of $\mathbf{e}:=\hat{\U}- \tilde{\U}\ge 0$, we see that there exists a sufficiently large $M^{'}>M$ such that \begin{equation}
	\begin{aligned}
		P_{k} e_{k} &=-f_{k}\left(x, t, \tilde{u}_{k},[\tilde{\U}]_{\mathrm{i}_{k}},[\hat{\U}]_{\mathrm{d}_{k}}\right)+f_{k}\left(x, t, \hat{u}_{k},[\hat{\U}]_{\mathrm{i}_{k}},[\tilde{\U}]_{\mathrm{d}_{k}}\right) \\
		& \le M\left( \left|\tilde{u}_{k}-\hat{u}_{k}\right|^{2}+\left|[\tilde{\U}]_{\mathrm{i}_{k}}-[\hat{\U}]_{\mathrm{i}_{k}}\right|^{2}+\left|[\hat{\U}]_{\mathrm{d}_{k}}-[\tilde{\U}]_{\mathrm{d}_{k}}\right|^{2} \right)^{\frac{1}{2}} \\
		&\le M^{'} e_{k}+M^{'} \sum_{j \neq k} e_{j} \text { in } \Omega_{T}, \\
	{B}_{k} e_{k} &=0 \text { on } S_{T}, \quad e_{k}(x, 0)=0 \text { in } \Omega .
	\end{aligned}
\end{equation}
By Theorem \ref{bj}, $\mathbf{e}\le 0$. Thus, $\tilde{\U}\equiv \hat{\U}$. Therefore, we see that $\tilde{\U}\equiv \hat{\U}$ is a solution to \eqref{1}. Assume that $\mathbf{w} \in<\underline{\U},\bar{\U}>$ is a solution to \eqref{1}. Then $\mathbf{w}=\mathcal{F}(\mathbf{w},\mathbf{w})$. By a similar discussion as in the proof of monotonic properties of $\{\underline{\U}^{i}\}$ and $\{\bar{\U}^{i}\}$, we can show that $\underline{\U}^{i}\le \mathbf{w} \le \bar{\U}^{i}$ for $i=1,2,3\cdots$. Thus, $\tilde{\U}\le \mathbf{w} \le \hat{\U}$, which implies that $\tilde{\U}\equiv \mathbf{w} \equiv \hat{\U}.$  
  \end{proof}	
 
 Hereinbelow, we establish the upper and lower solutions method for the problem \eqref{14}.
 
\begin{Theorem}\label{212}
	Assume that $(\bar{\U},\underline{\U})$ is a pair of coupled upper and lower solutions of \eqref{14}. If $\F$ satisfies the Lipschitz condition in $<\underline{\boldsymbol{\psi}},\bar{\boldsymbol{\psi}}>,$ i.e., there exists $M>0$ so that 
	\begin{equation}\label{s46}
		\left| f_{k}(x,t,\U)-f_{k}(x,t,\V) \right| \le M\left| \U-\V \right|  
	\end{equation}
	for $(x,t)\in V_{T}$, $\U, \V \in <\underline{\boldsymbol{\psi}},\bar{\boldsymbol{\psi}}>$ and $1\le k \le m$, where $\underline{\boldsymbol{\psi}}$ and $\bar{\boldsymbol{\psi}}$ be defined as in  Definition \ref{32}. Then $\bar{\U}\ge \underline{\U}$ on $V_{T}$.
\end{Theorem} 	

\begin{proof}
	Let $w_{k}:=\overline{u}_{k}-\underline{u}_{k}$; then $w_{k}$ satisfies $\tilde{P}_{k}w_{k}+\sum\limits_{j=1}^{m} a_{k{j}} w_{j} \ge 0$ on $V_{T}$, where $a_{kj}$ satisfies \eqref{38}. By Theorem \ref{1.4} , we see that $w_{k}\ge 0$ on $V_{T}$ for $1\le k \le m$.		
\end{proof}
  
 In the following, we give the upper and lower solutions method for the problem \eqref{14}.  
 \begin{Theorem}\label{sx2}
 	Let $(\bar{\U},\underline{\U})$ be pair of coupled upper and lower solutions of \eqref{14}, $\underline{\boldsymbol{\psi}}$ and $\bar{\boldsymbol{\psi}}$ be defined as Theorem \ref{212} and $\F$ satisfy the condition of Theorem \ref{212}. Assume that $\F$ has mixed qusimonotonicity property in $<\underline{\U},\bar{\U}>$, and $\F$ satisfies \eqref{s46} for all $(x,t)\in V_{T}$, $\U, \V \in <\underline{\U},\bar{\U}>$ and $1\le k \le m$. Then there exist two monotone sequences $\{\underline{\U}^{i}\}_{i=1}^{\infty}$ and $\{\overline{\U}^{i}\}_{i=1}^{\infty}$ satisfying 
 	\begin{equation}\label{321}
 		\underline{\U} \leqslant \underline{\U}^{i} \leqslant \underline{\U}^{i+1} \leqslant \bar{\U}^{i+1} \leqslant \bar{\U}^{i} \leqslant \bar{\U} \text { for all } i \geqslant 1, \lim\limits_{i\to \infty} \underline{\U}^{i}=\lim\limits_{i\to \infty} \overline{\U}^{i}=\U
 	\end{equation}
 	and $\U$ is the unique solution to \eqref{14} in the order interval $<\underline{\U},\overline{\U}>$.
 \end{Theorem} 	
 
 \begin{proof}
 	By Theorem \ref{212}, we have $\underline{\U}\le \bar{\U}$. It is well-known that for any given $\V,\W\in <\underline{\U},\overline{\U}>$ , the problem 
 	\begin{equation}
 		\begin{cases} \tilde{P}_{k} u_{k}+M u_{k}=f_{k}\left(x, t, v_{k},[\V]_{\mathrm{i}_{k}},[\mathbf{w}]_{\mathrm{d}_{k}}\right)+M v_{k} & \text { on } V_{T}, \\ u_{k}(x, 0)=\varphi_{k}(x) & \text { on } V \end{cases}
 	\end{equation}
 	admits a unique solution $u_{k}$, $1\le k \le m$.  Thus we can define an operator $\mathcal{F}$ by $\U=\mathcal{F}(\V,\mathbf{w})$ and 
 	\begin{equation}
 		\underline{\U}^{1}=\mathcal{F}(\underline{\U}, \bar{\U}), \quad \bar{\U}^{1}=\mathcal{F}(\bar{\U}, \underline{\U}), \quad \underline{\U}^{i+1}=\mathcal{F}\left(\underline{\U}^{i}, \bar{\U}^{i}\right),~ \text { and } \bar{\U}^{i+1}=\mathcal{F}\left(\bar{\U}^{i}, \underline{\U}^{i}\right).
 	\end{equation}
 
 By virtue of Theorem \ref{1.4}, by a similar discussion as in the proof of Theorem \ref{sx}, we obtain \eqref{321}.
 \end{proof}

\section{Stability of nontrivial steady-state solutions}
In this section, we give the proof of Theorems \ref{1.1}-\ref{m2}.
 \subsection{The proof of Theorem \ref{1a}}	
  	 	Let $f_{1}(u,v)=u(a_1-b_1u-c_1v)$ and $f_{2}(u,v)=v(a_2-b_2u-c_2v)$. Clearly, $\F(u,v)=(f_{1},f_{2})$ is qusimonotone nonincreasing in $\R^{+}\times \R^{+}$. Take $$M_{1}=\max\{\frac{a_1}{b_1}, \max_{\bar{\Omega}} u(x,0) \} ~\text{and}~ M_{2}=\max\{\max_{\bar{\Omega}}v(x,0), \frac{a_2}{c_2} \}.$$ In view of $u_{0}\ge 0$ and $v_{0}\ge 0$, it is easy to check that $(M_1, M_2)$ and $(0,0)$ are upper and lower solutions to \eqref{20220904-3} whether $B_{i}$ $(i=1,2)$ satisfy \eqref{1,4} or \eqref{1,5}. By Theorem \ref{sx}, \eqref{20220904-3} admits a unique solution $(u^{N},v^{N})$ so that $(0,0) \le (u^{N},v^{N}) \le (M_1, M_2)$ on $D:=\bar{\Omega}\times[0,+\infty)$ when $B_{i}$ $(i=1,2)$ satisfy \eqref{1,4}, and a unique solution $(u^{D},v^{D})$ so that $(0,0) \le (u^{D},v^{D}) \le (M_1, M_2)$ on $D$ when $B_{i}$ $(i=1,2)$ satisfy \eqref{1,5}. Furthermore, if ${u}_{0}\not\equiv 0$ and ${v}_0\not \equiv0$ on $\bar{ \Omega}$, then by Theorem \ref{bj}, we see that 
  	 	\begin{equation}\label{z}
  	 		u>0~ \text{and } v>0~ \text{on}~ \bar{\Omega}\times(0,+\infty)
  	 	\end{equation} 
   	when $B_{i}$ $(i=1,2)$ satisfy \eqref{1,4}, and that 
   	\begin{equation}\label{z1}
   		u>0~ \text{and } v>0~ \text{on}~ {\Omega}\times(0,+\infty)
   	\end{equation} 
  when $B_{i}$ $(i=1,2)$ satisfy \eqref{1,5}.
   	
   	We now complete the proof of Theorem \ref{1.1}.
  
 \subsection{The competition system with initial value and Neumann boundary} 
 In this subsection, we proof Theorem \ref{m} via Theorems \ref{4.1}-\ref{4.4}. 
  
 Consider the problem \eqref{20220904-3} with $B_{i}$ $(i=1,2)$ satisfying \eqref{1,4} and denote $(u,v):=(u^{N},v^{N})$. By \eqref{z} we have 
  	\begin{equation}\label{431}
  		\left\{\begin{array}{lll}
  			u_{t}-d_{1}\Delta_{\Omega_1} u\le u(a_1-b_1u), & x \in \Omega, & t>0, \\
  			v_{t}-d_{2}\Delta_{\Omega_2} v\le v(a_2-c_2v), & x \in \Omega, & t>0, \\
  			\frac{\partial u}{\partial_{\Omega_1} n}=\frac{\partial v}{\partial_{\Omega_2} n}=0, & x \in \partial \Omega, & t>0, \\
  			u=u_{0} \geqslant 0, \quad v=v_{0} \geqslant 0, & x \in \bar{\Omega}, & t=0,
  		\end{array}\right.
  	\end{equation}
 Consider the following problem
\begin{equation}\label{2}
	\left\{\begin{array}{lll}
		U_{t}-d_{i}\Delta_{\Omega_{i}} U= u(a_i-\beta_i U), & x \in \Omega, & t>0, \\
		\frac{\partial U}{\partial_{\Omega} n}=0, & x \in \partial \Omega, & t>0, \\
		U\equiv u_{i} , & x \in \bar{\Omega}, & t=0,
	\end{array}\right.
\end{equation}
where $i=1,2$, $u_{1}\equiv u_{0}$, $\beta_1\equiv b_1$, $u_{2}\equiv v_{0}$ and $\beta_2\equiv c_2$.
By \cite[Theorem 5.6]{HL}, we see that the problem \eqref{2} admits a unique nonnegative global solution $U(x,t)$ with $i=1$, and a unique nonnegative global solution $V(x,t)$ with $i=2$. By \cite[Theorem 4.1]{HL}, we see that $u(x,t)\le U(x,t)$ and $v(x,t)\le V(x,t)$ for $x\in \bar{\Omega}$, $t>0$. By \cite[Theorem 5.6]{HL} , we see that 
\begin{equation}\label{20220922-3}
	\lim\limits_{t\to \infty}U(x,t) =\frac{a_1}{b_1}\;\mbox{and}\;
	\lim\limits_{t\to \infty}V(x,t) =\frac{a_2}{c_2}.
\end{equation}
 Thus, we know that  
 \begin{equation}\label{20220924-4}
 	\limsup\limits_{t\to \infty}u(x,t) \le \frac{a_1}{b_1}~ \mbox{and}~  \limsup\limits_{t\to \infty}v(x,t) \le \frac{a_2}{c_2}.
 \end{equation} 

We next characterize the stability domain of the constant solutions of \eqref{20220904-3} with $B_{i}$ $(i=1,2)$ satisfying \eqref{1,4}. 

\begin{Theorem}\label{4.1}
	If $\frac{a_1}{a_2}< \frac{b_1}{b_2}$ and $\frac{a_1}{a_2}< \frac{c_1}{c_2}$, then $\lim\limits_{t\to \infty} (u(x,t),v(x,t))=(0,\frac{a_2}{c_2})$.
\end{Theorem}
\begin{proof}	 	
In view of \eqref{20220924-4}, for any given $0<\epsilon<\min\{\frac{c_1a_2}{c_2}-a_1,\frac{b_1a_2}{b_2}-a_1\}$, we can find $t_0$ such that 
\begin{equation}\label{20220924-5}
	u(x,t)<\frac{a_1 + \epsilon}{b_1},\;v(x,t)<\frac{a_2 +\epsilon}{c_2}\;\mbox{ uniformly for}\;x\in \bar{\Omega},\;t\ge t_0.
\end{equation}
Note that $\frac{a_1}{a_2}< \frac{c_1}{c_2}$ and $\frac{a_1}{a_2}< \frac{b_1}{b_2}$, we can choose sufficiently small $\sigma>0$ so that 
\begin{equation}\label{1,}
	\min\{a_{2}-\frac{c_2}{c_1}(a_1 + \epsilon),a_{2}-\frac{b_2}{b_1}(a_1 + \epsilon),b_{1}u(x,t_0),	b_{1}\min_{x\in\Omega}u(x,t_0),c_2 \min_{x\in\Omega}v(x,t_0)\}>\sigma,
\end{equation}
Then we can choose a sufficiently small $\beta>0$ such that \begin{equation}\label{2,}
	\beta \le \frac{c_1}{c_2}\sigma+\epsilon, ~\beta(a_2 -\sigma)\le \sigma[a_2-\sigma-\frac{b_2}{b_1}(a_1 +\epsilon)],\text{and}~ a_2\ge \beta.
\end{equation}

Set 
\begin{equation}\label{q}
	q=\frac{c_1}{c_2}(a_2+\epsilon)+\sigma-a_1,
\end{equation}
$$ 	\bar{u}(t)=\frac{(a_1 +\epsilon)e^{-\beta(t-t_0)}}{b_1},~\underline{u}(t)=\frac{\sigma e^{-q(t-t_0)}}{b_1},
$$
$$
 	\bar{v}(t)=\frac{a_2+\epsilon e^{-\beta(t-t_0)}}{c_2}\;\mbox{and}\; \underline{v}(t)=\frac{a_2-(a_2-\sigma)e^{-\beta(t-t_0)}}{c_2}.
$$
  	Direct calculations yield that 
  	 \begin{equation} 
  		\begin{aligned}
  			\bar{u}\left(a_{1}-b_{1} \bar{u}-c_{1} \underline{v}\right)&=\frac{\left(a_{1}+\varepsilon\right) e^{-\beta\left(t-t_{0}\right)}}{b_{1}}\left[a_{1}-\left(a_{1}+\varepsilon\right) e^{-\beta\left(t-t_{0}\right)}-c_{1} \frac{a_{2}-\left(a_{2}-\sigma\right) e^{-\beta\left(t-t_{0}\right)}}{c_{2}}\right]\\
  			&=\frac{\left(a_{1}+\varepsilon\right) e^{-\beta\left(t-t_{0}\right)}}{b_{1}}\left\{a_{1}-c_{1} \frac{a_{2}}{c_{2}}+\left[c_{1} \frac{a_{2}-\sigma}{c_{2}}-\left(a_{1}+\varepsilon\right)\right] e^{-\beta\left(t-t_{0}\right)}\right\}\\
  			&\overset{\eqref{1,}}{\leqslant} \frac{\left(a_{1}+\varepsilon\right) e^{-\beta\left(t-t_{0}\right)}}{b_{1}}\left[a_{1}-c_{1} \frac{a_{2}}{c_{2}}+c_{1} \frac{a_{2}-\sigma}{c_{2}}-\left(a_{1}+\varepsilon\right)\right]\\
  			&=\frac{\left(a_{1}+\varepsilon\right) e^{-\beta\left(t-t_{0}\right)}}{b_{1}}\left(-\frac{c_{1}}{c_{2}} \sigma-\varepsilon\right)\overset{\eqref{2,}}{\leqslant}-\beta \frac{\left(a_{1}+\varepsilon\right) e^{-\beta\left(t-t_{0}\right)}}{b_{1}}\\
  			&=\bar{u}_{t}-d_1\Delta_{\Omega_1} \bar{u},
  		\end{aligned}
  	\end{equation} 
  \begin{equation}
  	\begin{aligned}
  		\underline{v}\left(a_{2}-b_{2} \bar{u}-c_{2} \underline{v}\right)&=\frac{a_{2}-\left(a_{2}-\sigma\right) e^{-\beta\left(t-t_{0}\right)}}{c_{2}}\left[a_{2}-\sigma-\frac{b_{2}}{b_{1}}\left(a_{1}+\varepsilon\right)\right] e^{-\beta\left(t-t_{0}\right)}\\
  		&\stackrel{\eqref{1,}}{\geqslant} \frac{\sigma}{c_{2}}\left[a_{2}-\sigma-\frac{b_{2}}{b_{1}}\left(a_{1}+\varepsilon\right)\right] e^{-\beta\left(t-t_{0}\right)}\\
  		&\stackrel{\text { \eqref{2,} }}{\geqslant} \frac{a_{2}-\sigma}{c_{2}} \beta e^{-\beta\left(t-t_{0}\right)}=\underline{v}_{t}-d_2 \Delta_{\Omega_2} \underline{v},
  	\end{aligned}
  \end{equation}	
  	
 \begin{equation}
 	\begin{aligned}
 		  \underline{u}\left(a_{1}-b_{1} \underline{u}-c_{1} \bar{v}\right) 
 		&=\frac{\sigma e^{-q\left(t-t_{0}\right)}}{b_{1}}\left[a_{1}-\sigma e^{-q\left(t-t_{0}\right)}-c_{1} \frac{a_{2}+\varepsilon e^{-\beta\left(t-t_{0}\right)}}{c_{2}}\right] \\
 		& \geqslant \frac{\sigma e^{-q\left(t-t_{0}\right)}}{b_{1}}\left[a_{1}-\sigma-\frac{c_{1}}{c_{2}}\left(a_{2}+\varepsilon\right)\right] \\
 		&=\frac{-q\sigma e^{-q(t-t_0)}}{b_1} \overset{\eqref{q}}{\geqslant} \underline{u}_{t}-d_1\Delta_{\Omega_1} \underline{u},
 	\end{aligned}
 \end{equation} 	
  	
  	\begin{equation}
  		\begin{aligned}
  			\bar{v}\left(a_{2}-b_{2} \underline{u}-c_{2} \bar{v}\right) &=\frac{a_{2}+\varepsilon e^{-\left(t-t_{0}\right)}}{c_{2}}\left[-b_{2} \frac{\sigma e^{-q \left(t-t_{0}\right)}}{b_{1}}-\varepsilon e^{-\beta\left(t-t_{0}\right)}\right] \\
  			& \leqslant-\frac{a_{2}}{c_{2}} \varepsilon e^{-\beta\left(t-t_{0}\right)} \\
  			& \overset{\eqref{2,}}{\leqslant} -\frac{\varepsilon}{c_{2}} \beta e^{-\beta\left(t-t_{0}\right)}=\bar{v}_{t}-d_2 \Delta_{\Omega_2} \bar{v},
  		\end{aligned}
  	\end{equation}
for $t\ge t_0$ and $x\in \bar{\Omega}$, and $B\bar{u}=B\bar{v}=B\underline{u}=B\underline{v}=0$ for $t\ge t_0$ and $x\in \partial{\Omega}$. By \eqref{1,}, we see that  $$\bar{u}(t_0)=\frac{a_1+\epsilon}{b_1}> u(x,t_0)> \underline{u}(t_0)=\frac{\sigma}{b_1},$$ $$\bar{v}(t_0)=\frac{a_2+\epsilon}{c_2}>v(x,t_0)>\underline{v}(t_0)=\frac{\sigma}{c_2}.$$	Thus, by Definition \ref{3}, $(\bar{u},\bar{v})$, $(\underline{u},\underline{v})$ is a pair of coupled upper and lower solutions to \eqref{41}. By Theorem \ref{21}, we have $(\underline{u},\underline{v})\le(u,v) \le(\bar{u},\bar{v})$ on $\bar{\Omega}\times [t_0,+\infty)$. Since
 \begin{equation}
	0=\lim\limits_{t\to \infty} \underline{u}(t)= \lim\limits_{t\to \infty}\bar{u}(t)=0,
\end{equation}
and
  \begin{equation}
 	\frac{a_2}{c_2}=\lim\limits_{t\to \infty} \underline{v}(t)= \lim\limits_{t\to \infty}\bar{v}(t)=\frac{a_2}{c_2},
 \end{equation} 	
this implies that $\lim\limits_{t\to +\infty}u(x,t)=0~ \text{and} \lim\limits_{t \to +\infty}v(x,t)=\frac{a_2}{c_2}$. 
\end{proof}

\begin{Theorem}\label{4.2}
	If $\frac{b_1}{b_2}<\frac{a_1}{a_2},$ $\frac{c_1}{c_2}<\frac{a_1}{a_2}$, then 
	\begin{equation}
	\lim\limits_{t\to \infty} (u(x,t),v(x,t))=(\frac{a_1}{b_1},0).
	\end{equation}
\end{Theorem}
\begin{proof}
Due to \eqref{20220924-4},	for any  $0<\epsilon<\min\{\frac{a_1c_2}{c_1}-a_2,\frac{a_1b_2}{b_1}-a_2\}$, we can find $t_0>0$ such that \eqref{20220924-5} holds. Then we can find sufficiently small $\sigma>0$ such that 
\begin{equation}\label{2,1}
	\sigma<\min\{a_1-\frac{c_1}{c_2}(a_2+\epsilon),~a_1-\frac{b_1}{b_2}(a_2+\epsilon),~c_2 \min_{\bar{\Omega}} v(x,t_0),~b_1 \min_{\bar{\Omega}} u(x,t_0)\}.
\end{equation}
And we can choose $0<\beta$ such that
\begin{equation}\label{2,2}
	\beta\le \frac{b_2}{b_1} \sigma+\epsilon,~(a_1-\sigma)\beta\le \sigma[a_1-\sigma-\frac{c_1}{c_2}(a_2+\epsilon)].
\end{equation}

Define $$
	\bar{u}(t)=\frac{a_1+\epsilon e^{-a_1(t-t_0)}}{b_1},~\underline{u}(t)=\frac{a_1-(a_1-\sigma)e^{-\beta(t-t_0)}}{b_1},$$
	$$\bar{v}(t)=\frac{(a_2+\epsilon)e^{-\beta(t-t_0)}}{c_2}~\mbox{and}\;\underline{v}(t)=\frac{\sigma e^{-q(t-t_0)}}{c_2}
$$
where $q=\sigma +\frac{b_2}{b_1}(a_1+\epsilon) -a_2$.
Direct calculations show that
\begin{equation}
	\begin{aligned}
		\bar{u}\left( a_1-b_1\bar{u}-c_1 \underline{v}\right)&=\frac{a_1+\epsilon e^{-a_1(t-t_0)}}{b_1}\left[ -\epsilon e^{-a_1(t-t_0)}-\frac{c_1}{c_2} \sigma e^{-q(t-t_0)} \right] \\
		&\le -\frac{a_1}{b_1}\epsilon e^{-a_1(t-t_0)}=\bar{u}_{t}-d_1 \Delta_{\Omega_1} \bar{u},  
	\end{aligned}
\end{equation} 

\begin{equation}
	\begin{aligned}
		\underline{u}\left(a_{1}-b_{1} \underline{u}-c_{1} \bar{v}\right) &=\frac{a_{1}-\left(a_{1}-\sigma\right) e^{-\beta\left(t-t_{0}\right)}}{b_{1}}\left\{a_{1}-\left[a_{1}-\left(a_{1}-\sigma\right) e^{-\beta\left(t-t_{0}\right)}\right]-c_{1} \frac{\left(a_{2}+\epsilon\right) e^{-\beta\left(t-t_{0}\right)}}{c_{2}}\right\} \\
		&=\frac{a_{1}-\left(a_{1}-\sigma\right) e^{-\beta\left(t-t_{0}\right)}}{b_{1}}\left[\left(a_{1}-\sigma\right) e^{-\beta\left(t-t_{0}\right)}-\frac{c_{1}}{c_{2}}\left(a_{2}+\epsilon\right) e^{-\beta\left(t-t_{0}\right)}\right] \\
		&\overset{\eqref{2,1}}{\geqslant}  \frac{\sigma}{b_{1}}\left[a_{1}-\sigma-\frac{c_{1}}{c_{2}}\left(a_{2}+\epsilon\right)\right] e^{-\beta\left(t-t_{0}\right)} \\
		&\overset{\eqref{2,2}}{\geqslant}  \frac{a_{1}-\sigma}{b_{1} }\beta e^{-\beta\left(t-t_{0}\right)}=\underline{u}_{t}-d_1 \Delta_{\Omega_1} \underline{u},
	\end{aligned}
\end{equation}
\begin{equation}
	\begin{aligned}
		\bar{v}\left(a_{2}-b_{2} \underline{u}-c_{2} \bar{v}\right) &=\frac{\left(a_{2}+\epsilon\right) e^{-\beta\left(t-t_{0}\right)}}{c_{2}}\left[a_{2}-b_{2} \frac{a_{1}}{b_{1}}+\frac{b_{2}}{b_{1}}\left(a_{1}-\sigma\right) e^{-\beta\left(t-t_{0}\right)}-\left(a_{2}+\epsilon\right) e^{-\beta\left(t-t_{0}\right)}\right] \\
		&\overset{\eqref{2,1}}{\leqslant}  \frac{\left(a_{2}+\epsilon\right) e^{-\beta\left(t-t_{0}\right)}}{c_{2}}\left(-\frac{b_{2}}{b_{1}} \sigma-\epsilon\right) \\
		&\overset{\eqref{2,2}}{\leqslant}-\beta \frac{a_{2}+\epsilon}{c_{2}} e^{-\beta\left(t-t_{0}\right)}=\bar{v}_{t}-d_2 \Delta_{\Omega_2} \bar{v},
	\end{aligned}
\end{equation}
\begin{equation}
	\begin{aligned}
		\underline{v}(a_2-b_2\bar{u}-c_2 \underline{v})&= \frac{\sigma e^{-q(t-t_0)}}{c_2}\left\lbrace a_2-\frac{b_2}{b_1}(a_1+\epsilon e^{-a_1(t-t_0)})-\sigma e^{-q(t-t_0)} \right\rbrace \\		
	&\ge	\frac{\sigma e^{-q(t-t_0)}}{c_2}[a_2-\frac{b_2}{b_1}(a_1+\epsilon)-\sigma] \\
	&	= \frac{\sigma} {c_2}(-q)e^{-q(t-t_0)}=\underline{v}_{t}-d_2 \Delta_{\Omega_2} \underline{v},
	\end{aligned}
\end{equation}
for $t\ge t_0$ and $x\in \bar{\Omega}$, and $B_1 \bar{u}=B_2\bar{v}=B_1\underline{u}=B_2\underline{v}=0$ for $t\ge t_0$ and $x\in \partial{\Omega}$. By \eqref{2,1}, we see that  $$\bar{u}(t_0)=\frac{a_1+\epsilon}{b_1}> u(x,t_0)> \underline{u}(t_0)=\frac{\sigma}{b_1}, \bar{v}(t_0)=\frac{a_2+\epsilon}{c_2}>v(x,t_0)>\underline{v}(t_0)=\frac{\sigma}{c_2}.$$	Thus, by Definition \ref{3}, $(\bar{u},\bar{v})$, $(\underline{u},\underline{v})$ is a pair of coupled upper and lower solutions to \eqref{41}. By Theorem \ref{21}, we have $(\underline{u},\underline{v})\le(u,v) \le(\bar{u},\bar{v})$ on $\bar{\Omega}\times [t_0,+\infty)$. Since 
\begin{equation}
	\frac{a_1}{b_1}=\lim\limits_{t\to \infty} \underline{u}(t)= \lim\limits_{t\to \infty}\bar{u}(t)=\frac{a_1}{b_1},
\end{equation}
and 
\begin{equation}
	0=\lim\limits_{t\to \infty} \underline{v}(t)= \lim\limits_{t\to \infty}\bar{v}(t)=0,
\end{equation} 	
we see that $\lim\limits_{t\to +\infty}u(x,t)=\frac{a_1}{b_1}~ \text{and} \lim\limits_{t \to +\infty}v(x,t)=0$. 
\end{proof}

\begin{Theorem}\label{4.3}
If	$\frac{c_{1}}{c_{2}}<\frac{a_{1}}{a_{2}}<\frac{b_{1}}{b_{2}}$,
then 
$$\lim\limits_{t\to +\infty}\left( u(x,t),v(x,t) \right) =(\xi,\eta).$$
\end{Theorem}
\begin{proof}
By \eqref{20220924-4} and  $\frac{c_{1}}{c_{2}}<\frac{a_{1}}{a_{2}}<\frac{b_{1}}{b_{2}},$ for any given $0<\epsilon<\min\{\frac{c_2a_1}{c_1}-a_2,\frac{a_2b_1}{b_1}-a_1\}$, we can choose $t_0$ so that \eqref{20220924-5} holds. Then there exists $\sigma>0$ such that
\begin{equation}\label{b}
	\sigma<\min \left\{b_1 \xi, c_2 \eta, c_{2} {\min\limits_{\bar{\Omega}} v(x,t_0)} , a_{2}-\frac{b_{2}}{b_{1}}\left(a_{1}+\epsilon\right), a_{1}-\frac{c_{1}}{c_{2}}\left(a_{2}+\epsilon\right), b_{1} {\min\limits_{\bar{\Omega}} u\left(x, t_{0}\right)} \right\} \text {. }
\end{equation}
Recall that $\xi,\eta$ are defined by \eqref{20220924}. Clearly, we can choose a sufficiently small $q<\min\{q_1,q_2,q_3,q_4\}$, where 
\begin{equation}\label{43}
	\begin{aligned}
		& q_1=\frac{\sigma\left[a_{2}-\sigma-\frac{b_{2}}{b_{1}}\left(a_{1}+\epsilon\right)\right]}{c_2 \eta-\sigma},
		q_2=\frac{b_{1} \xi \left(\frac{c_{1}}{c_{2}} \sigma+\epsilon\right)}{a_{1}+\epsilon-b_{1}\xi},\\
		&q_3=\frac{\sigma\left[a_{1}-\sigma-\frac{c_{1}}{c_{2}}\left(a_{2}+\epsilon\right)\right]}{b_1 \xi-\sigma}, \text { and } q_4=\frac{c_2\eta(\frac{b_2}{b_1}\sigma+\epsilon)}{a_{2}+\epsilon-c_2 \eta}.
	\end{aligned}
\end{equation}
Define 
\begin{equation}
	\begin{aligned}
		&\bar{u}(t)=\frac{b_{1} \xi+\left(a_{1}+\epsilon-b_{1} \xi\right) e^{-q\left(t-t_{0}\right)}}{b_{1}}, \underline{u}(t)=\frac{b_{1} \xi-\left(b_{1} \xi-\sigma\right) e^{-q\left(t-t_{0}\right)}}{b_{1}}, \\
		&\bar{v}(t)=\frac{c_{2} \eta+\left(a_{2}+\epsilon-c_{2} \eta\right) e^{-q\left(t-t_{0}\right)}}{c_{2}}, \underline{v}(t)=\frac{c_{2} \eta-\left(c_{2} \eta-\sigma\right) e^{-q\left(t-t_{0}\right)}}{c_{2}}.
	\end{aligned}
\end{equation}
Direct calculations show that 
\begin{equation}
	\begin{aligned}
		\bar{u}\left(a_{1}-b_{1} \bar{u}-c_{1} \underline{v}\right) &= \frac{b_{1}\xi+\left( a_{1}+\epsilon-b_{1}\xi \right) e^{-q\left(t-t_{0}\right)}}{b_{1}} 	\bigg	\{ a_{1}- b_{1} \xi-c_{1} \eta+ \\ &\left. \left.\left[\frac{c_{1}}{c_{2}} \left( c_{2} \eta-\sigma\right)-\left(a_{1}+\epsilon-b_{1} \xi\right) \right] e^{-q\left(t-t_{0}\right)}\right\} \right.\\
		& \leqslant \xi\left[\frac{c_{1}}{c_{2}} \left(c_{2}\eta-\sigma\right)-\left(a_{1}+\epsilon-b_{1} \xi\right)\right] e^{-q\left(t-t_{0}\right)}=-\xi\left(\frac{c_{1}}{c_{2}} \sigma+\epsilon\right) e^{-q\left(t-t_{0}\right)} \\
		&\overset{\eqref{43}}{\leqslant} -\frac{\left(a_{1}+\epsilon-b_{1} \xi\right)}{b_{1}} q e^{-q\left(t-t_{0}\right)}=\bar{u}_{t}-d_{1} \Delta_{\Omega_{1}} \bar{u},
	\end{aligned}
\end{equation}
\begin{equation}
	\begin{aligned}
		\bar{v}\left(a_{2}-b_{2} \underline{u}-c_{2} \bar{v}\right) &=\frac{c_{2} \eta+\left(a_{2}+\epsilon-c_{2} \eta\right) e^{-q\left(t-t_{0}\right)}}{c_{2}}\left(-\frac{b_{2}}{b_{1}} \sigma-\epsilon\right) e^{-q\left(t-t_{0}\right)} \\
		& \leqslant-\eta\left(\frac{b_{2}}{b_{1}} \sigma+\epsilon\right) e^{-q\left(t-t_{0}\right)} \overset{\eqref{43}}{\leqslant}-\left(a_{2}+\epsilon-c_{2} \eta\right) q e^{-q\left(t-t_{0}\right)} \\
		&=\bar{v}_{t}-d_{2} \Delta_ \Omega \bar{v},
	\end{aligned}
\end{equation}
\begin{equation}
	\begin{aligned}
		\underline{u}\left(a_{1}-b_{1} \underline{u}-c_{1} \bar{v}\right) &=\frac{b_{1} \xi-\left(b_{1} \xi-\sigma\right) e^{-q\left(t-t_{0}\right)}}{b_{1}} \bigg \{ a_{1}-\left[b_{1} \xi-\left(b_{1} \xi-\sigma\right) e^{-q\left(t-t_{0}\right)}\right] \\
		&\left.\left. -\frac{c_{1}}{c_{2}}\left[c_{2} \eta+\left(a_{2}+\varepsilon-c_{2} \eta\right) e^{-q\left(t-t_{0}\right)}\right]\right\} \right.\\
		&=\frac{b_{1} \xi-\left(b_{1} \xi-\sigma\right) e^{-q\left(t-t_{0}\right)}}{b_{1}}\left[a_{1}-\sigma-\frac{c_{1}}{c_{2}}\left(a_{2}+\epsilon\right)\right] e^{-q\left(t-t_{0}\right)} \\
		&\overset{\eqref{b}}{\geqslant}  \frac{\sigma}{b_{1}}\left[a_{1}-\sigma-\frac{c_{1}}{c_{2}}\left(a_{2}+\epsilon\right)\right] e^{-q\left(t-t_{0}\right)} \\
		&\overset{\eqref{43}}{\geqslant}  \frac{\left(b_{1} \xi-\sigma\right) q}{b_{1}} e^{-q\left(t-t_{0}\right)}=\underline{u}_{t}-d_1 \Delta_{\Omega_{1}} \underline{u},
	\end{aligned}
\end{equation}
\begin{equation}
	\begin{aligned}
		\underline{v}\left(a_{2}-b_{2} \overline{u}-c_{2} \underline{v}\right) &=\frac{c_{2} \eta-\left(c_{2}\eta-\sigma \right) e^{-q\left(t-t_{0}\right)}}{c_{2}}\left\{a_{2}-b_{2} \frac{b_{1} \xi+\left(a_{1}+\epsilon-b_{1} \xi\right) e^{-q\left(t-t_{0}\right)}}{b_{1}}\right.\\
		&\left.-\left[c_{2} \eta-\left(c_{2} \eta-\sigma\right) e^{-q\left(t-t_{0}\right)}\right]\right\} \\
		&\overset{\eqref{b}}{\geqslant}  \frac{\sigma}{c_{2}}\left[a_{2}-\frac{b_{2}}{b_{1}}\left(a_{1}+\epsilon\right)-\sigma\right] e^{-q\left(t-t_{0}\right)} \\
		&\overset{\eqref{43}}{\geqslant} \frac{c_{2} \eta-\sigma}{c_{2}} q e^{-q\left(t-t_{0}\right)}=\underline{v}_{t}-d_{2} \Delta_{\Omega_{2}} \underline{v},
	\end{aligned}
\end{equation}
for $t\ge t_0$ and $x\in \bar{\Omega}$, and $B_{1}\bar{u}=B_{2}\bar{v}=B_{1}\underline{u}=B_{2}\underline{v}=0$ for $t\ge t_0$ and $x\in \partial{\Omega}$. By \eqref{b}, we see that  $$\bar{u}(t_0)=\frac{a_1+\epsilon}{b_1}> u(x,t_0)> \underline{u}(t_0)=\frac{\sigma}{b_1}, \bar{v}(t_0)=\frac{a_2+\epsilon}{c_2}>v(x,t_0)>\underline{v}(t_0)=\frac{\sigma}{c_2}.$$	Thus, by Definition \ref{3}, $(\bar{u},\bar{v})$, $(\underline{u},\underline{v})$ is a pair of upper and lower solutions to \eqref{41}. By Theorem \ref{21}, we have $(\underline{u},\underline{v})\le(u,v) \le(\bar{u},\bar{v})$ on $\bar{\Omega}\times [t_0,+\infty)$. Since
\begin{equation}
	\xi=\lim\limits_{t\to \infty} \underline{u}(t)= \lim\limits_{t\to \infty}\bar{u}(t)=\xi,
\end{equation}
and
\begin{equation}
	\eta=\lim\limits_{t\to \infty} \underline{v}(t)= \lim\limits_{t\to \infty}\bar{v}(t)=\eta,
\end{equation} 	
this implies that $\lim\limits_{t\to +\infty}u(x,t)=\xi ~ \text{and} \lim\limits_{t \to +\infty}v(x,t)=\eta$. 	
\end{proof}

\begin{Theorem}\label{4.4}
	Suppose that $\frac{b_{1}}{b_{2}}<\frac{a_{1}}{a_{2}}<\frac{c_{1}}{c_{2}}$. If $\xi<u_{0}(x)<\frac{a_{1}}{b_{1}}$ and $0<v_{0}(x)<\eta$ on $\bar{ \Omega}$, then $\lim _{t \rightarrow \infty}(u(x, t), v(x, t))=\left(\frac{a_{1}}{b_{1}}, 0\right)$; If $0<u_0<\xi$ and $\eta<v_0(x)< \frac{a_2}{c_2}$ on $\bar{ \Omega}$, then $\lim _{t \rightarrow \infty}(u(x, t), v(x, t))=\left(0,\frac{a_{2}}{c_{2}} \right)$.  
\end{Theorem}
\begin{proof}
	We first deal with the case that $\xi<u_{0}(x)<\frac{a_{1}}{b_{1}}$ and $0<v_{0}(x)<\eta$. Define 
	$$	R_{1}:=\left\{(x, y), 0<y< {\eta}, \frac{a_{2}-c_{2} y}{b_{2}}<x<\frac{a_{1}-c_{1} y}{b_{1}}\right\}.$$
	For any sufficiently small $\epsilon^{'}$ we may find  two points $P=\left(\xi_{1}, \eta_{1}\right),~ Q=\left(\xi_{2}, \eta_{2}\right)\in	R_{1}$ satisfying $\xi_{1}\le u_0$, $\eta_{1}\ge v_0$,   $\xi_{2}>\frac{a_1}{b_1}-\epsilon^{'}>\xi$, and $\eta_{2}<\epsilon^{'}$. Thus we can find $\delta>0$ so that \begin{equation}
		\min \left\{f_{1}(\xi, \eta) ;(\xi, \eta) \in \overline{P Q}\right\} \geqslant \delta,~\max \left\{f_{2}(\xi, \eta) ;(\xi, \eta) \in \overline{P Q} \right\}\leqslant-\delta, 
	\end{equation} 
where $\overline{PQ}$ denotes the straight line $\overline{PQ}$ with end points $P,Q$. Set 
$$
p(t)=\xi_{2}+\left(\xi_{1}-\xi_{2}\right) e^{-\varepsilon t}, q(t)=\eta_{2}+\left(\eta_{1}-\eta_{2}\right) e^{-\varepsilon t}~(t\ge 0),
$$
where $\varepsilon \leqslant \min \left\{\delta\left|\xi_{1}-\xi_{2} \right|^{-1}, \delta\left|\eta_{1}-\eta_{2}\right|^{-1} \right\} .$ Then we see that $(p,q)$ lies on $\overline{PQ}$ for $t\ge 0$. It is easy to check that 
\begin{equation}
	\begin{aligned}
		&p^{\prime}(t)=-\varepsilon\left(\xi_{1}-\xi_{2}\right) e^{-\varepsilon t} \leqslant \delta \leqslant f_{1}(p, q), \\
		&q^{\prime}(t)=-\varepsilon\left(\eta_{1}-\eta_{2}\right) e^{-\varepsilon t} \geqslant-\delta \geqslant f_{2}(p, q).
	\end{aligned}
\end{equation}
Let $\bar{u}=\frac{a_1}{b_1}$, $\bar{v}=q(t)$, $\underline{u}=p(t)$, and $\underline{v}=0$; then $(\bar{u},\bar{v})$ and $(\underline{u},\underline{v})$ satisfies \begin{equation}\label{2sx}
	\begin{cases}\bar{u}_{t}-d_1\Delta_{\Omega_1} \bar{u} \geqslant \bar{u}(a_1-b_1\bar{u}-c_1 \underline{v}), & x \in \Omega, \quad t>0, \\ \underline{v}_{t}-d_2 \Delta_{\Omega_2} \underline{v} \leqslant \underline{v}(a_2-b_2 \underline{v}-c_2 \bar{u}), & x \in \Omega, \quad t>0, \\ \underline{u}_{t}-d_1 \Delta_{\Omega_1} \underline{u} \leqslant \underline{u}(a_1-b_1 \underline{u}-c_1 \bar{v}), & x \in \Omega,\quad t>0, \\ 
		\bar{v}_{t}-d_2 \Delta_{\Omega_2} \bar{v} \geqslant \bar{v}(a_2-c_2\bar{v}-b_2 \underline{u}), & x \in \Omega,\quad t>0.\end{cases}
\end{equation} In view of $p(0)=\xi_{1} \leqslant u_{0} \leqslant \frac{a_{1}}{b_{1}}, \quad 0 \leqslant v_{0} \leqslant \eta_{1}=q(0)$, by Theorem \ref{21}, we see that 
\begin{equation}
	(p(t), 0) \leqslant(u(x, t), v(x, t)) \leqslant\left(\frac{a_{1}}{b_{1}}, q(t)\right), t>0, \quad x \in \bar{\Omega}.
\end{equation}
It follows that 
\begin{equation}
	\begin{aligned}
		&\frac{a_{1}}{b_{1}}-\epsilon^{'}<\xi_{2}=\lim _{t \rightarrow+\infty} p(t) \leqslant \liminf _{t \rightarrow+\infty} u(x, t) \leqslant \limsup _{t \rightarrow+\infty} u(x, t) \leqslant \frac{a_{1}}{b_{1}} \text {, }\\
		&0 \leqslant \liminf\limits_{t \rightarrow+\infty} v(x, t) \leq \limsup\limits _{t \rightarrow+\infty} v(x, t) \leq \lim\limits _{t \rightarrow+\infty} q(t)=\eta_{2}\le \epsilon^{'} \text {. }
	\end{aligned}
\end{equation}
Thus we get \begin{equation}
	\lim _{t \rightarrow+\infty}(u(x, t), v(x, t))=\left(\frac{a_{1}}{b_{1}}, 0\right)\;\mbox{uniformly for}\;x\in\bar{\Omega}.
\end{equation}

We next deal with another case. Define
$$R_{2}:=\left\{(x, y), 0<x<{\xi}, \frac{a_{1}-b_{1} x}{c_{1}}<y <\frac{a_{2}-b_{2} x}{c_{2}}\right\}.$$
For any sufficiently small $\epsilon_{0}>0$, we may
find two points $P=\left(\xi_{1}, \eta_{1}\right),~ Q=\left(\xi_{2}, \eta_{2}\right)\in R_{2} $
satisfying $u_0 \le \xi_{1}$, $\eta_{1} \le v_0$,  $\eta_{2}>\frac{a_2}{c_2}-\epsilon_0$, $\xi_{2}<\epsilon_0$. Thus we can find $\delta>0$ so that 
\begin{equation}
	\max \left\{f_{1}(\xi, \eta) ;(\xi, \eta) \in \overline{P Q}\right\} \leqslant -\delta,~\min \left\{f_{2}(\xi, \eta) ;(\xi, \eta) \in \overline{P Q} \right\}\geqslant \delta, 
\end{equation} 
where $\overline{PQ}$ denotes the straight line $\overline{PQ}$ with end points $P,Q$. Set 
$$
p(t)=\xi_{2}+\left(\xi_{1}-\xi_{2}\right) e^{-\varepsilon t}, q(t)=\eta_{2}+\left(\eta_{1}-\eta_{2}\right) e^{-\varepsilon t}~(t\ge 0),
$$
where $\varepsilon \leqslant \min \left\{\delta\left|\xi_{1}-\xi_{2} \right|^{-1}, \delta\left|\eta_{1}-\eta_{2}\right|^{-1} \right\} .$ Then we see that $(p,q)$ lies on $\overline{PQ}$ for $t\ge 0$. It is easy to check that 
\begin{equation}
	\begin{aligned}
		&p^{\prime}(t)=-\varepsilon\left(\xi_{1}-\xi_{2}\right) e^{-\varepsilon t} \geqslant -\delta \geqslant f_{1}(p, q), \\
		&q^{\prime}(t)=-\varepsilon\left(\eta_{1}-\eta_{2}\right) e^{-\varepsilon t} \leqslant \delta \leqslant f_{2}(p, q).
	\end{aligned}
\end{equation}
Let $\bar{u}=p(t)$, $\bar{v}=\frac{a_2}{c_2}$, $\underline{u}=0$, and $\underline{v}=q(t)$; then $(\bar{u},\bar{v})$ and $(\underline{u},\underline{v})$ satisfies \eqref{2sx}. By virtue of $ \eta_{1}=q(0) \leqslant v_{0} \leqslant \frac{a_{2}}{c_{2}},~ 0 \leqslant u_{0} \leqslant p(0)=\xi_{1}$, by Theorem \ref{21}, we see that 
\begin{equation}
	(0, q(t))) \leqslant(u(x, t), v(x, t)) \leqslant\left(p(t),\frac{a_{2}}{c_{2}} \right), t>0, \quad x \in \bar{\Omega}.
\end{equation}
It follows that 
\begin{equation}
	\begin{aligned}
		&0 \leqslant \liminf _{t \rightarrow+\infty} u(x, t) \leqslant \limsup _{t \rightarrow+\infty} u(x, t) \leqslant \lim _{t \rightarrow+\infty} p(t)=\xi_{2}\le \epsilon_0 \text {, }\\
		&\frac{a_{2}}{c_{2}}-\epsilon_0 \le \lim _{t \rightarrow+\infty}q(t) \leqslant \liminf\limits_{t \rightarrow+\infty} v(x, t) \leq \limsup\limits _{t \rightarrow+\infty} v(x, t) \leq \frac{a_{2}}{c_{2}} \text {. }
	\end{aligned}
\end{equation}
Thus we get \begin{equation}
	\lim _{t \rightarrow+\infty}(u(x, t), v(x, t))=(0, \frac{a_2}{c_2})
\end{equation} uniformly for $x\in\bar{\Omega}$.
\end{proof}

Combining Theorems \ref{4.1}, \ref{4.2}, \ref{4.3} and \ref{4.4}, we may obtain Theorem \ref{m}.
\subsection{The competition system with initial value and Dirichlet boundary}
In this subsection, we study the problem \eqref{41} with $B_{i}$ $(i=1,2)$ satisfying \eqref{1,4}. We get the existence and characterization of the steady-state solutions.

Denote $(u,v):=(u^{D},v^{D})$.
Consider the following problem
\begin{equation}\label{12}
	\left\{\begin{array}{lll}
		U_{t}-d_{i}\Delta_{\Omega_{i}} U= U(a_i-\beta_i U), & x \in \Omega, & t>0, \\
		U=0, & x \in \partial \Omega, & t>0, \\
		U\equiv u_{i} , & x \in \bar{\Omega}, & t=0,
	\end{array}\right.
\end{equation}
where $i=1,2$, $u_{1}\equiv u_{0}$, $\beta_1\equiv b_1$, $u_{2}\equiv v_{0}$ and $\beta_2\equiv c_2$. By \cite[Theorem 5.5]{HL}, we conclude that the problems \eqref{12} admit a unique nonnegative global solution $\check{u}(x,t), \check{v}(x,t)$ with $i=1,2$, respectively.
 
 Recall that $\lambda_{0,i}$ the smallest eigenvalue of the eigenvalue problem \eqref{451} $(i=1,2)$. 
 
We show that the trivial solution $(0,0)$ is globally asymptotically stable, when $a_1\le \lambda_{0,1} d_1$ and $a_2\le \lambda_{0,2} d_2$ by the following theorem. 
 \begin{Theorem}\label{Th}
Suppose that $u_0, v_0\ge 0$ on $\bar{ \Omega}$. Then $\lim\limits_{t\to +\infty} u(x,t)=0$ when $a_1\le \lambda_{0,1} d_1$; $\lim\limits_{t\to \infty} v(x,t)= 0$, when $a_2\le \lambda_{0,2} d_2$.
 \end{Theorem}
 	\begin{proof}
It follows from \cite[Theorem 4.1]{HL} that $0\le u\le \check{u}$ and $0\le v\le 
 \check{v}$. In view of \cite[Theorem 5.5]{HL}, then $\lim\limits_{t\to +\infty} \check{u}(x,t)=0$ provided that $a_1\le \lambda_{0,1} d_1$ and $\lim\limits_{t\to +\infty} \check{v}(x,t)=0$ provided that  $a_2 \le \lambda_{0,2} d_2$. Therefore, we see that if $a_1\le \lambda_{0,1} d_1$, then $\lim\limits_{t\to +\infty} {u}(x,t)=0$; if $a_2\le \lambda_{0,2} d_2$, then  $\lim\limits_{t\to +\infty} {v}(x,t)=0$.   
 	\end{proof} 
 
 We establish the existence of nontrivial steady-state solutions of the problem \eqref{20220904-3} with $B_{i}$ $(i=1,2)$ satisfying \eqref{1,5} and unsability of the trivial solution $(0,0)$ when either $a_1>\lambda_{0,1} d_1$ or $a_2>\lambda_{0,2} d_2$ in the following theorem.
 \begin{Theorem}
 	If $a_1>\lambda_{0,1 }d_1$, $a_2\le \lambda_{0,2} d_2$ and $u_0\ge,\not\equiv 0$, then $\lim\limits_{t\to +\infty} \left( {u}(x,t),{v}(x,t)\right) =\left( s_1(x),0\right)$; If $a_1\le \lambda_{0,1} d_1$, $a_2> \lambda_{0,2} d_2$ and $v_0\ge,\not\equiv 0$, then $\lim\limits_{t\to +\infty} \left( {u}(x,t),{u}(x,t)\right) =\left( 0,s_2(x)\right)$.
 \end{Theorem}	
\begin{proof}
	As $a_1>\lambda_{0,1} d_1$, $a_2\le \lambda_{0,2} d_2$, by \cite[Theorem 5.5]{HL}, we see that $\lim\limits_{t\to +\infty} \check{u}(x,t)=s_1(x)$ and $\lim\limits_{t\to \infty}\check{v}(x,t)=0$. We apply \cite[Theorem 4.1]{HL} to conclude that $u(x,t)\le \check{u}(x,t)$, $v(x,t)\le \check{v}(x,t)$.  Thus, it's easily seen that
	 \begin{equation}\label{i1}
		\limsup\limits_{t\to \infty} u(x,t)\le s_{1}(x)
	\end{equation} 
and $\lim\limits_{t \rightarrow \infty} v(x,t)=0$. For any $0<\epsilon<\epsilon_{0}:=\frac{a_1- \lambda_{0,1} d_1}{c_1}$, we can find $T_{1}=T_{1}(\epsilon)>0$ such that $v\le \epsilon$ for $x\in\bar{\Omega}$ and $t\ge T_1$. Thus, the problem 
	\begin{equation}\label{452}
		\left\{\begin{array}{lll}
			U_{t}-d_{1}\Delta_{\Omega_1} U= U(a_1-b_1 U-c_1\epsilon), & x \in \Omega, & t>T_1, \\
			U=0, & x \in \partial \Omega, & t>T_1, \\
			U\equiv u_{0}(x) , & x \in \bar{\Omega}, & t=T_1
		\end{array}\right.
	\end{equation}
	admits a unique positive solution $\underline{u}_{\epsilon}(x,t)$ satisfying $\lim\limits_{t\to\infty} \underline{u}_{\epsilon}(x,t)=s_{1,\epsilon}(x)$ by \cite[Theorem 5.5 ]{HL}, where $s_{1,\epsilon}(x)$ is the unique positive solution to problem 
	\begin{equation}\label{20220924-6}
		\begin{cases}
			-d_{1}\Delta_{\Omega_1} s_{1,\epsilon}= s_{1,\epsilon}\left(a_1-b_1 s_{1,\epsilon}-c_1\epsilon \right) ~\text{ in }~\Omega,\\
			s_{1,\epsilon}=0~\text{ on }~\partial{\Omega}.
		\end{cases}
	\end{equation}
 Then we use \cite[Theorem 3.8]{HL} to deduce that 
 $$0\le s_{1,\epsilon}(x)\le \frac{a_{1}-c_{1}\epsilon}{b_1}~\text{for}~x\in \bar{ \Omega}~\text{and}~\epsilon\in (0,\epsilon_{0})$$ and 
 hence that $s_{1,\epsilon}(x)$ is bounded with respect to $0\le \epsilon\le \epsilon_0$ for all $x\in \bar{ \Omega}$. Thus there exists a sequence $\{\epsilon_{n}\}_{n=1}^{\infty}$ satisfying $0<\epsilon_{n}<\epsilon_{0}$ and $\epsilon_{n}\to 0$ as $n\to +\infty$ and $\bar{S}(x)$ so that $s_{1,\epsilon_{n}}(x)\to \bar{S}(x)$ for $x\in \bar{\Omega}$. Letting $n\to \infty$ in \eqref{20220924-6} with $\epsilon$ replaced by $\epsilon_{n}$, we conclude that $\bar{S}$ satisfies \eqref{20220904-1} with $i=1$.
By the uniqueness of the positive solution to \eqref{20220904-1} with $i=1$, we see that $\bar{S}(x)=s_{1}(x)$. It follows from \cite[Theorem 4.1]{HL} that $u(x,t)\ge \underline{u}_{\epsilon}(x,t)$ for $x\in\bar{ \Omega}$ and $t\ge T_1$. This implies that that \begin{equation}
	\liminf_{t \rightarrow+\infty} u(x,t)\ge s_{1,\epsilon_n}(x)~\text{for}~x\in \bar{ \Omega}. 
\end{equation}
Letting $n \to \infty$ in the above inequality, we see that $\liminf\limits_{t \rightarrow \infty} u(x,t)\ge s_{1}(x)$. Combining this with \eqref{i1}, we deduce that $\lim\limits_{t\to \infty} u(x,t)=s_{1}(x)$. 
 Therefore, we obtain $\lim\limits_{t\to\infty}u(x,t)=s_{1}(x).$ 

By a similar argument as above, we see that if $a_{1}\le \lambda_{0,1} d_1$ and $a_{2}> \lambda_{0,2} d_2$, then $\lim\limits_{t\to+\infty} (u,v)=(0,s_{2})$.
\end{proof}

We next study the case $a_1 > \lambda_0 d_1$ and $a_2>\lambda_0 d_2$, we establish a sufficient condition for the species $u$ and $v$ coexist.

{\bf Proof of Theorem \ref{m1} (iii):}
	 In view of \cite[Theorem 4.1]{HL} and Theorem 2.2, we see that $0< u\le \check{u}$ and $0< v\le \check{v}$ for $\overline{\Omega}\times(0,+\infty)$. As $a_1 > \lambda_{0,1} d_1$ and $a_2>\lambda_{0,2} d_2$,  it's easily seen that  $$\lim\limits_{t\to +\infty} \check{u}(x,t)=s_{1}(x)\;\mbox{ and}\;\lim\limits_{t\to +\infty} \check{v}(x,t)=s_{2}(x)$$
	 by \cite[Theorem 5.5]{HL}. 
	Therefore, for $0<\epsilon<\min \left\{ \frac{b_1}{a_1b_2}(a_2-\lambda_{0,2} d_2)-1,\frac{c_2}{a_2c_1}(a_1-\lambda_{0,1} d_1)-1\right\}$, there exists $T_2>0$ so that
	\begin{equation}\label{460}
		u(x,T_2)\le (1+\epsilon)s_1(x),~v(x,T_2)\le (1+\epsilon) s_{2}(x),\;\mbox{for}\;x\in \bar{\Omega}).
	\end{equation}

Let $\phi_{i}$ be the eigenvector of \eqref{451} corresponding to $\lambda_{0,i}$, $i=1,2$. By \cite{HL}, we can assume that $\phi_{i}>0$ in $\Omega$, $i=1,2$.  Thus, by \eqref{k1},  there exist sufficiently small $\delta>0$ so that
and
 \begin{equation}\label{459}
	\delta\phi_{1},~\delta\phi_{2}(x) \le \min\left\lbrace  u(x,T_2),~v(x,T_2),~E,~F   \right\rbrace~\text{for}~x\in \bar{ \Omega},
\end{equation} 	
where
\begin{equation*}
	E:=\frac{1}{c_2}\left[ a_2-\lambda_{0,2} d_2-\left(1+\epsilon \right) \frac{a_1}{b_1}b_2 \right]>0,~F:=\frac{1}{b_1} \left[ a_1-\lambda_{0,1} d_1-\left(1+\epsilon\right)\frac{a_2}{c_2}c_1 \right]>0.
\end{equation*}
	Set 
	$$\bar{U}:=(1+\epsilon)s_1,~ \bar{V}:=(1+\epsilon)s_2,~ \underline{U}:=\delta \phi_{1},~ \underline{V}:=\delta \phi_{2}.$$ 
	Thanks to $s_{1}\ge 0$ and $\underline{V}\ge 0$, we see that 
	\begin{equation}\label{461}
		\bar{U}_{t}-d_1 \Delta_{\Omega_{1}} \bar{U}=-d_1 (1+\epsilon)\Delta_{\Omega_{1}} s_{1}=(1+\epsilon)s_1(a_1-b_1s_1)\ge \bar{U}(a_1-b_1\bar{U}-c_1 \underline{V}).
	\end{equation} 

Since $s_1\le \frac{a_1}{b_1}$ on $\bar{ \Omega}$, by \eqref{459}, we deduce that
\begin{equation}
	\underline{V}_{t}-d_2\Delta_{\Omega_{2}} \underline{V}= \delta\lambda_0 d_{2}\phi_{2}\le \underline{V}\left(a_2-b_2\bar{U}-c_2 \underline{V} \right) .
\end{equation}
In view of $s_2\le \frac{a_2}{c_2}$, by \eqref{459}, we see that 
\begin{equation}
	\underline{U}_{t}-d_1 \Delta_{\Omega_{1}} \underline{U}=-d_1\Delta_{\Omega_{1}}(\delta\phi_{1})=d_1\delta \lambda_0 \phi_{1}\le \underline{U}\left(a_1 -b_1\underline{U}-c_1 \bar{V} \right). 
\end{equation}
Thanks to $\underline{U}\ge 0$ and $s_2 \ge 0$, we conclude that 
\begin{equation}
	\bar{V}_{t}-d_2 \Delta_{\Omega_{2}} \bar{V}=-d_2 \left( 1+\epsilon \right)\Delta_{\Omega_{2}} s_2=(1+\epsilon)s_{2}(a_2 -c_2 s_2 ) \ge \bar{V}\left( a_2-b_2\underline{U}-c_2 \bar{V}\right).  
\end{equation}

Let $(\hat{u},\tilde{v})$ be the unique nonnegative global solution to \eqref{20220904-3} with  $B_{1}u=u=0,~ B_{2}v=v=0$, $u_0=\bar{U}$ and $v_0=\underline{V}$ and $(\tilde{u},\hat{v})$ be the unique nonnegative global solution to \eqref{20220904-3} with  $B_{1}u=u=0, B_{2}v=v=0$, $u_0=\underline{U}$ and $v_0=\bar{V}$.
Let $W:=\hat{u}-\bar{U}$ and $K:=\underline{V}-\tilde{v}$; then by \eqref{461} we see that 
\begin{equation}\label{4312}
	\left\{\begin{array}{lll}
		W_{t}-d_{1}\Delta_{\Omega_{1}} W\le \left[a_1 -b_1 \left(\hat{u}+\bar{U} \right)-c_1 \tilde{v}  \right] W+c_1 \bar{U} K , & x \in \Omega, & t>0, \\
		K_{t}-d_{2}\Delta_{\Omega_{2}} K \le \left[a_2 -c_2\left(\underline{V}+\tilde{v} \right)-b_2 \bar{U}  \right]K+b_2 \tilde{v} W , & x \in \Omega, & t>0, \\
		 W= K=0, & x \in \partial \Omega, & t>0, \\
		W=0 , \quad K=0 , & x \in \bar{\Omega}, & t=0.
	\end{array}\right.
\end{equation}
Thus, by Theorem \ref{bj}, we see that $W\le 0$ and $K\le 0$. This implies that $\hat{u}\le \bar{U}$ and $\underline{V}\le \tilde{v}$ for $t\ge 0$ and $x\in \bar{ \Omega}$. By Theorem 4.1 in \cite{HL}, we see that 
for any $\epsilon^{'}>0$, 
\begin{equation}
	 \hat{u}(x,t) \ge 
	 \hat{u}(x,t+\epsilon^{'}) ~\text{and}~
\tilde{v}(x,t)\le	  \tilde{v}(x,t+\epsilon^{'}) .  
\end{equation}
By a similar argument as above, we may deduce that $\tilde{u} \ge \underline{U}$, $\hat{v}\le \bar{V}$ and  for any $\epsilon^{'}>0$,
\begin{equation}
\tilde{u}(x,t) \le \tilde{u}(x,t+\epsilon^{'})~\text{and}~
	\hat{v}(x,t)	\ge \hat{v}(x,t+\epsilon^{'})  , 
\end{equation}
$t\ge 0$, $x\in \bar{\Omega}$. This implies that $\hat{u}(x,t)$ and $\hat{v}(x,t)$ are nonincreasing and $\tilde{u}$ and $\tilde{v}$ are nondecreasing with respect to t. Thus, we can define 
\begin{equation}\label{1b}
	\left(\bar{s}(x),\underline{s}(x),\bar{r}(x),\underline{r}(x)\right):=\lim_{t\to+\infty}\left(\hat{u}(x,t),\tilde{u}(x,t),\hat{v}(x,t),\tilde{v}(x,t) \right)   .
\end{equation}
By Theorem \ref{21}, \eqref{460} and \eqref{459}, we conclude that \begin{equation}\label{2b}
	\tilde{u}(x,t)\le u(x,t+T_2) \le \hat{u}(x,t),~\tilde{v}(x,t)\le v(x,t+T_2)\le \hat{v}(x,t)
\end{equation}
for $t\ge 0$ and $x\in\bar{\Omega}$.

Combining \eqref{1b} with \eqref{2b}, we  obtain \eqref{458}.

Now, we suppose the conditions \eqref{d} and \eqref{tj} hold. At this time, we have $\Delta_{\Omega_{1}}=\Delta_{\Omega_{2}}$, $\phi_{1}\equiv\phi_{2}=\phi$ and $\lambda_{0,1}=\lambda_{0,2}$. Rewrite the three as $\Delta_{\Omega}$, $\phi$, $\lambda$. Let $(\hat{U},\tilde{V})$ be the unique nonnegative global solution to \eqref{20220904-3} with  $B_{1}u=u=0,~ B_{2}v=v=0$, $u_0=\bar{U}$ and $v_0=\delta \phi$ and $(\tilde{U},\hat{V})$ be the unique nonnegative global solution to \eqref{20220904-3} with  $B_{1}u=u=0, B_{2}v=v=0$, $u_0=\delta \phi$ and $v_0=\bar{V}$. Then we see that
\begin{equation}
	\begin{cases}
		\hat{u}_{t}-d_1 \Delta_{\Omega}\hat{u}=\hat{u}\left(a_1 -b_1\hat{u}-c_1\tilde{v}
\right)~(x,t)\in \Omega\times (0,+\infty) ,\\
\tilde{u}_{t}-d_1\Delta_{\Omega}\tilde{u}=\tilde{u}\left( a_1-b_1\tilde{u}-c_1\hat{v}\right),~(x,t)\in \Omega\times (0,+\infty), \\
\hat{v}_{t}-d_2\Delta_{\Omega} \hat{v}=\hat{v}\left(a_2-b_2\tilde{u}-c_2 \hat{v} \right),~(x,t)\in \Omega\times (0,+\infty),\\
\tilde{v}_{t}-d_2\Delta_{\Omega} \tilde{v}=\tilde{v}\left(a_2-b_2\hat{u}-c_2 \tilde{v} \right),~(x,t)\in \Omega\times (0,+\infty).  
	\end{cases}
\end{equation}
This implies that for any $t_1>0$ and $T>0$, 
 \begin{equation}\label{t1}
 	\frac{\hat{u}(x,T)-\hat{u}(x,t_1)}{T}-\frac{1}{T} \int_{t_1}^{T} d_1\Delta_{\Omega} \hat{u} dt=\frac{1}{T}\int_{t_1}^{T}\hat{u}\left( a_1-b_1\hat{u}-c_1\tilde{v} \right) dt,
 \end{equation} 	
	\begin{equation}\label{t2}
		\frac{\tilde{u}(x,T)-\tilde{u}(x,t_1)}{T}-\frac{1}{T} \int_{t_1}^{T} d_1\Delta_{\Omega} \tilde{u} dt=\frac{1}{T}\int_{t_1}^{T}\tilde{u}\left( a_1-b_1\tilde{u}-c_1\hat{v} \right) dt,
	\end{equation}
\begin{equation}\label{t3}
	\frac{\hat{v}(x,T)-\hat{v}(x,t_1)}{T}-\frac{1}{T} \int_{t_1}^{T} d_2 \Delta_{\Omega} \hat{v} dt=\frac{1}{T}\int_{t_1}^{T}\hat{v}\left( a_2-b_2\tilde{u}-c_2\hat{v} \right) dt,
\end{equation}
	\begin{equation}\label{t4}
	\frac{\tilde{v}(x,T)-\tilde{v}(x,t_1)}{T}-\frac{1}{T} \int_{t_1}^{T} d_2 \Delta_{\Omega} \tilde{v} dt=\frac{1}{T}\int_{t_1}^{T}\tilde{v}\left( a_2-b_2\hat{u}-c_2\tilde{v} \right) dt.
\end{equation}
Letting $T\to \infty$ in \eqref{t1}-\eqref{t4}, we deduce that 
\begin{equation}
	-d_1 \Delta_{\Omega} \bar{s}=\bar{s}\left(a_1-b_1\bar{s}-c_1\underline{r} \right), 
\end{equation}
\begin{equation}
	-d_1 \Delta_{\Omega} \underline{s}=\underline{s}\left(a_1-b_1\underline{s}-c_1\overline{r} \right), 
\end{equation}
\begin{equation}
	-d_2 \Delta_{\Omega} \bar{r}=\bar{r}\left(a_2-b_2\underline{s}-c_2\bar{r} \right), 
\end{equation}
\begin{equation}
	-d_2 \Delta_{\Omega} \underline{r}=\underline{r}\left(a_2-b_2\bar{s}-c_2\underline{r} \right), 
\end{equation}
for $x\in\Omega$, where $\bar{s}=\bar{s}(x;\tilde{\omega},\mu)$, $\underline{s}=\bar{s}(x;\tilde{\omega},\mu)$, $\bar{r}=\bar{r}(x;\tilde{\omega},\mu)$, $\underline{r}=\bar{r}(x;\tilde{\omega},\mu)$.

 Clearly, $\bar{s}=\underline{s}=\bar{r}=\underline{r}=0$ on $\partial\Omega$. Thus, $\left(\bar{s},\underline{r} \right) $, $\left(\underline{s}, \bar{r} \right) $ are solutions to \eqref{ss}

 By \eqref{1b} and \eqref{2b}, we know that $\underline{s} \le \bar{s} $ and $\underline{r} \le \bar{r}$. Let $S(x):=\bar{s}-\underline{s}$ and $R(x):=\bar{r}-\underline{r}$; then we see that 
\begin{equation}\label{469}
	\begin{aligned}
		&-d_{1} \Delta_{\Omega} S(x)=\left[a_{1}-b_{1}(\bar{s}+\underline{s})-c_{1} \underline{r}\right] S(x)+c_{1} \underline{s} R(x),\\
		&-d_{2} \Delta_{\Omega} R(x)=\left[a_{2}-c_{2}(\bar{r}+\underline{r})-b_{2} \underline{s}\right] R(x)+b_{2} \underline{r} S(x)
	\end{aligned}
\end{equation}
Multiplying the equalities in \eqref{469} by $\phi$, and by integration by parts, we obtain 
\begin{equation}
	\begin{aligned}
		&\left(\lambda_{0} d_{1}-a_{1}\right) \int_{\Omega} S \phi d \mu=\int_{\Omega}\left[-b_{1}(\bar{s}+\underline{s}) S \phi-c_{1} \underline{r}(x) S(x) \phi+c_{1} \underline{s} R \phi \right]d \mu, \\
		&\left(\lambda_{0} d_{2}-a_{2}\right) \int_{\Omega} R \phi d \mu=\int_{\Omega}\left[-c_{2}(\bar{r}+\underline{r}) R \phi-b_{2} \underline{s} R \phi+b_{2} \underline{r}(x) S \phi \right]d \mu.
	\end{aligned}
\end{equation}
Multiplying the first equation by $b_2$, the second by $c_1$ and adding yield,
 \begin{equation}
	b_{2} \int_{\Omega}\left[\lambda_{0} d_{1}-a_{1}+b_{1}(\bar{s}+\underline{s})\right] S \phi d \mu+c_{1} \int_{\Omega}\left[\lambda_{0} d_{2}-a_{2}+c_{2}(\bar{r}+\underline{r})\right] R \phi  d \mu=0
\end{equation}
 In view of \eqref{tj}, and the fact that $\bar{s}\ge \underline{s}$, $\underline{r}\le \bar{r}$, we see that
\begin{equation}
	b_{2} \int_{\Omega}\left[\left(\lambda_{0} d_{1}-a_{1}\right)+2 b_{1} \underline{s} \right] (\bar{s}-\underline{s}) \phi d \mu+c_{1} \int_{\Omega}\left[\lambda_{0} d_{2}-a_{2}+2 c_{2} \underline{r}\right] (\bar{r}-\underline{r}) \phi d \mu \leqslant 0 .
\end{equation}
Thus, by \eqref{tj} and the fact that $\phi>0$ in $\Omega$, we see that 
\begin{equation}\label{e}
	\bar{s}\equiv \underline{s}~\text{ and}~ \bar{r}\equiv \underline{r}.
\end{equation}

We now complete the proof of Theorem \ref{m1}.

\subsection{The competition system with initial value} 
In this subsection, we investigate the problem \eqref{41} and obtain the domain of attraction of the steady-state solutions. 

We first give the proof of Theorem \ref{1a}.
Take $$M_{3}=\max\{\frac{a_1}{b_1}, \max_{V} u(x,0) \} ~\text{and}~ M_{4}=\max\{\max_{V}v(x,0), \frac{a_2}{c_2} \}.$$ It is easy to check that $(M_3, M_4)$ and $(0,0)$ are upper and lower solutions to \eqref{41}. By Theorem \ref{sx2}, \eqref{41} admits a unique solution $(u^{*},v^{*})$ defined for all $t>0$ so that $(0,0) \le (u^{*},v^{*}) \le (M_3, M_4)$ on $V$. Furthermore, if $\tilde{u}_{0}\ge, \not\equiv 0$ and $\tilde{v}_0\ge, \not \equiv0$ on $V$, then by Theorem \ref{1.4}, we see that 
\begin{equation}\label{z2}
	u^{*}>0~ \text{and } v^{*}>0~ \text{on}~ V\times(0,+\infty).
\end{equation} 
We now complete the proof of Theorem \ref{m1}.

Denote $(u,v)=(u^{*},v^{*})$. 
By \eqref{z2} we have 
\begin{equation}\label{4}
	\left\{\begin{array}{lll}
		u_{t}-d_{1}\Delta_{V_1} u\le u(a_1-b_1u), & x \in V, & t>0, \\
		v_{t}-d_{2}\Delta_{V_2} v\le v(a_2-c_2v), & x \in V, & t>0, \\
		u=\tilde{u}_{0} \geqslant 0, \quad v=\tilde{v}_{0} \geqslant 0, & x \in V, & t=0,
	\end{array}\right.
\end{equation}
Now, we consider the following problem
\begin{equation}\label{22}
	\left\{\begin{array}{lll}
		U_{t}-d_{i}\Delta_{V_i} U= U(a_i-\beta_i U), & x \in V, & t>0, \\
		U\equiv u_{i} , & x \in V, & t=0,
	\end{array}\right.
\end{equation}
where $i=1,2$, $u_{1}\equiv \tilde{u}_{0}$, $\beta_1\equiv b_1$, $u_{2}\equiv \tilde{v}_{0}$ and $\beta_2\equiv c_2$.
By a similar discussion as in the proof of \cite[Theorem 5.7]{HL}, we deduce that the problem \eqref{22} admits a unique nonnegative solution $U(x,t)$ with $i=1$, and a unique nonnegative  solution $V(x,t)$ with $i=2$. It follows from \cite[Theorem 4.1]{HL} that 
$$u(x,t)\le U(x,t)\;\mbox{and}\;v(x,t)\le V(x,t)\;\mbox{for }\;x\in V,~ t>0.$$ 
By a similar argument as in the proof of \cite[Theorem 5.7]{HL}, we see that $$\lim\limits_{t\to \infty}U(x,t) =\frac{a_1}{b_1}\;\mbox{and}\;\lim\limits_{t\to \infty}V(x,t) =\frac{a_2}{c_2}.$$  
Then, it follows that  
\begin{equation}\label{20220924-7}
\limsup\limits_{t\to \infty}u(x,t) \le \frac{a_1}{b_1}\;\mbox{ and }\;\limsup\limits_{t\to \infty}v(x,t) \le \frac{a_2}{c_2}.
\end{equation}
 {\bf Proof of Theorem \ref{m2}:}
	The upper and lower solutions used in Theorem 1 are still applicable here. Thus,  by Theorem \ref{sx2}, we might obtain the desired conclusions. For completeness, we give the details.

We first deal with the case that $\frac{a_1}{a_2}< \frac{b_1}{b_2}$ and $\frac{a_1}{a_2}< \frac{c_1}{c_2}$.  For $0<\epsilon<\min\{\frac{c_1a_2}{c_2}-a_1,\frac{b_1a_2}{b_2}-a_1\}$, there exists $t_0>0$ such that 
\begin{equation}\label{20220924-8}
	u(x,t)<\frac{a_1 + \epsilon}{b_1}\;\mbox{and}\;v(x,t)<\frac{a_2 +\epsilon}{c_2}\; \mbox{uniformly for}\;x\in V,t\ge t_0.
\end{equation}
In this case, we could choose  
sufficiently small $\sigma>0$, $\beta>0$ satisfies \eqref{1,}, \eqref{2,}, respectively.

Define 
$$
\bar{u}(t)=\frac{(a_1 +\epsilon)e^{-\beta(t-t_0)}}{b_1},~\underline{u}(t)=\frac{\sigma e^{-q(t-t_0)}}{b_1},$$
$$
\bar{v}(t)=\frac{a_2+\epsilon e^{-\beta(t-t_0)}}{c_2}, \underline{v}(t)=\frac{a_2-(a_2-\sigma)e^{-\beta(t-t_0)}}{c_2}.
$$
where $q=\frac{c_1}{c_2}(a_2+\epsilon)+\sigma-a_1$.
A direct calculation yields that 
\begin{equation}\label{494}
	\left\{\begin{array}{lll}
		\bar{u}_t-d_1 \Delta_{V_1} \bar{u} \geqslant \bar{u}\left(a_1-b_1 \bar{u}-c_1 \underline{v}\right), & x \in V, & t>0\\
		\underline{v}_t-d_2 \Delta_{V_2} \underline{v} \leqslant \underline{v}\left(a_2-b_2 \bar{u}-c_2 \underline{v}\right), & x \in V, & t>0\\
		\underline{u}_t-d_1 \Delta_{V_1} \underline{u} \leqslant \underline{u}\left(a_1-b_1 \underline{u}-c_1 \bar{v}\right), & x \in V, & t>0\\
		\bar{v}_t-d_2 \Delta_{V_2} \bar{v} \geqslant \bar{v}\left(a_2-b_2 \underline{u}-c_2 \bar{v}\right), & x \in V, & t>0\\
		\bar{u}\left(t_0\right)>u\left(x, t_0\right)>\underline{u}\left(t_0\right), \bar{v} \left(t_0\right)>\bar{v}\left(x, t_0\right)>\underline{v}\left(t_0\right)~&x\in V.
	\end{array}\right.
\end{equation}
Thus by Definition \ref{34}, $(\bar{u},\bar{v})$, $(\underline{u},\underline{v})$ is a pair of upper and lower solutions to \eqref{41}. By Theorem \ref{212}, we have $(\underline{u},\underline{v})\le(u,v) \le(\bar{u},\bar{v})$ on $V \times [t_0,+\infty)$. Since
\begin{equation}
	0=\lim\limits_{t\to \infty} \underline{u}(t)= \lim\limits_{t\to \infty}\bar{u}(t)=0,
\end{equation}
and
\begin{equation}
	\frac{a_2}{c_2}=\lim\limits_{t\to \infty} \underline{v}(t)= \lim\limits_{t\to \infty}\bar{v}(t)=\frac{a_2}{c_2},
\end{equation} 	
this implies that $\lim\limits_{t\to +\infty}u(x,t)=0~ \text{and} \lim\limits_{t \to +\infty}v(x,t)=\frac{a_2}{c_2}$. 

Second we consider the case that 	$\frac{b_1}{b_2}<\frac{a_1}{a_2},$ $\frac{c_1}{c_2}<\frac{a_1}{a_2}$. For any  $0<\epsilon<\min\{\frac{a_1c_2}{c_1}-a_2,\frac{a_1b_2}{b_1}-a_2\}$, we can find $t_0>0$ such that \eqref{20220924-8} holds. Then we can find sufficiently small $\sigma,\beta>0$ satisfying \eqref{2,1}, \eqref{2,2}, respectively.

Now, we define
\begin{equation*}
\underline{u}_{1}(t)=\frac{a_1-(a_1-\sigma)e^{-\beta(t-t_0)}}{b_1},~\bar{v}_{1}(t)=\frac{(a_2+\epsilon)e^{-\beta(t-t_0)}}{c_2},~\underline{v}_{1}(t)=\frac{\sigma e^{-q(t-t_0)}}{c_2},
\end{equation*}
where $q=\sigma +\frac{b_2}{b_1}(a_1+\epsilon) -a_2$. It is easy to check that $(\underline{u}_{1},\underline{v}_{1})$ and $ (\bar{u}_{1},\bar{v}_{1})$ satisfy \eqref{494}. By Theorem \ref{212}, we have $(\underline{u}_{1},\underline{v}_{1})\le(u,v) \le(\bar{u}_{1},\bar{v}_{1})$ on $V \times [t_0,+\infty)$. Since 
\begin{equation}
	\frac{a_1}{b_1}=\lim\limits_{t\to \infty} \underline{u}(t)= \lim\limits_{t\to \infty}\bar{u}(t)=\frac{a_1}{b_1},
\end{equation}
and 
\begin{equation}
	0=\lim\limits_{t\to \infty} \underline{v}(t)= \lim\limits_{t\to \infty}\bar{v}(t)=0,
\end{equation} 	
we see that $\lim\limits_{t\to +\infty}u(x,t)=\frac{a_1}{b_1}~ \text{and} \lim\limits_{t \to +\infty}v(x,t)=0$.

Third we consider the case that 	
$	\frac{c_{1}}{c_{2}}<\frac{a_{1}}{a_{2}}<\frac{b_{1}}{b_{2}}$. For any given $0<\epsilon<\min\{\frac{c_2a_1}{c_1}-a_2,\frac{a_2b_1}{b_1}-a_1\}$, we can choose $t_0$ so that \eqref{20220924-8} holds. Due to $\frac{c_{1}}{c_{2}}<\frac{a_{1}}{a_{2}}<\frac{b_{1}}{b_{2}},$ we can choose sufficiently small $\sigma>0$ and $0<q<\min\{q_1,q_2,q_3,q_4\}$ satisfying \eqref{b} and \eqref{43}, respectively. 

Set 
\begin{equation}
	\begin{aligned}
		&\bar{u}_{2}=\frac{b_{1} \xi+\left(a_{1}+\varepsilon-b_{1} \xi\right) e^{-q\left(t-t_{0}\right)}}{b_{1}}, \underline{u}_{2}=\frac{b_{1} \xi-\left(b_{1} \xi-\sigma\right) e^{-q\left(t-t_{0}\right)}}{b_{1}}, \\
		&\bar{v}_{2}=\frac{c_{2} \eta+\left(a_{2}+\varepsilon-c_{2} \eta\right) e^{-q\left(t-t_{0}\right)}}{c_{2}}, \underline{v}_{2}=\frac{c_{2} \eta-\left(c_{2} \eta-\sigma\right) e^{-q\left(t-t_{0}\right)}}{c_{2}}.
	\end{aligned}
\end{equation}
It is easily seen that $(\underline{u}_{2},\underline{v}_{2})$ and $ (\bar{u}_{2},\bar{v}_{2})$ satisfy \eqref{494}. By Theorem \ref{212}, we have $(\underline{u}_{2},\underline{v}_{2})\le(u,v) \le(\bar{u}_{2},\bar{v}_{2})$ on $V \times [t_0,+\infty)$. Since
\begin{equation}
	\xi=\lim\limits_{t\to \infty} \underline{u}_{2}(t)= \lim\limits_{t\to \infty}\bar{u}_{2}(t)=\xi,
\end{equation}
and
\begin{equation}
	\eta=\lim\limits_{t\to \infty} \underline{v}_{2}(t)= \lim\limits_{t\to \infty}\bar{v}_{2}(t)=\eta,
\end{equation} 	
this implies that $\lim\limits_{t\to +\infty}u(x,t)=\xi ~ \text{and} \lim\limits_{t \to +\infty}v(x,t)=\eta$. 

Finally we consider the case that $\frac{b_{1}}{b_{2}}<\frac{a_{1}}{a_{2}}<\frac{c_{1}}{c_{2}}$. Step by step along the proof of Theorem \ref{4.4} (using Theorem \ref{212} instead of Theorem \ref{21}), we could deduce that $\lim\limits _{t \rightarrow \infty}(u(x, t), v(x, t))=\left(\frac{a_{1}}{b_{1}}, 0\right)$ when $\xi<u_{0}(x)<\frac{a_{1}}{b_{1}}$ and $0<v_{0}(x)<\eta$ on $V$ and that
$\lim\limits _{t \rightarrow \infty}(u(x, t), v(x, t))=\left(0,\frac{a_{2}}{c_{2}} \right)$ when $0<u_0<\xi$ and $\eta<v_0(x)< \frac{a_2}{c_2}$ on $V$.   

\section{Examples}

 We give some examples with numerical experiments to demonstrate Theorems \ref{1.1} and \ref{m1}.
\begin{example}\label{e1}
	Choose a graph $\overline{\Omega}$ satisfying $\Omega=\{x_1 ,x_2, x_3 \}$ and $\partial{\Omega}=\{x_4, x_5 \}$ whose vertices are linked as the following figure with a weight $\omega$ satisfying 
	$
	\omega_{x_i x_j}=\begin{cases}
		1,~x_i \sim x_j,\\
		0,~x_i \not= x_j
	\end{cases}
	$
	for $i,j=1,2,\cdots,5$. Let $d_1=d_2=1$ and $\mu(x)=\sum\limits_{y\in \bar{ \Omega}:y\sim x} \omega_{xy}$.
\end{example}	
	\begin{center}
		\begin{tikzpicture}[>=stealth, scale=1.2]
			
			\draw(0,0)node{\hspace{-.1cm}$\circ$}--(1,0)node{$\bullet$}--(2,0)node{$\bullet$}--(3,0)node{\hspace{.1cm}$\circ$};
			\draw(1,0)--(1.5,0.867)node{$\bullet$}--(2,0);
			\draw(0,0)node[below]{$x_4$}(1,0)node[below]{\hspace{-.1cm}$x_1$}(2,0)node[below]{\hspace{.3cm}$x_3$}(3,0)node[above]{$x_5$} (1.5,0.867)node[above]{$x_2$};
		\end{tikzpicture}
	\end{center}

	Suppose that $(u(x,t),v(x,t))$ is the unique global positive solution of \eqref{20220904-3} with
	$$\Delta_{\Omega_{1}}u(x,t)=\sum_{y\in \bar{\Omega}}[u(y,t)-u(x,t)] \frac{\omega_{xy}}{\mu(x)},$$
	$$\Delta_{\Omega_{2}}v(x,t)=\sum_{y\in \bar{\Omega}}[v(y,t)-v(x,t)] \frac{\omega_{xy}}{\mu(x)},$$
	 $$\frac{\partial u}{\partial_{\Omega_{1}} n}(x,t)=\sum_{y\in \Omega} [u(x,t)-u(y,t)] \frac{\omega_{xy}}{\mu(x)}~\text{on}~\partial\Omega\times[0,+\infty),$$
	 $$\frac{\partial v}{\partial_{\Omega_{2}} n}(x,t)=\sum_{y\in \Omega} [v(x,t)-v(y,t)] \frac{\omega_{xy}}{\mu(x)}~\text{on}~\partial\Omega\times[0,+\infty),$$
	  $u_0(x_1)=7$, $u_0(x_2)=6$, $u_0(x_3)=5$, $v_0(x_1)=4$, $v_0(x_2)=3$, and $v_0(x_3)=2$.

If we choose $a_1=1$, $b_1=c_1=2$, $a_2=1$, $b_2=1$ and $c_2=1$, then $\frac{a_1}{a_2}< \frac{b_1}{b_2}$ , $\frac{a_1}{a_2}< \frac{c_1}{c_2}$, $\frac{a_2}{c_2}=1$. Thus, by Theorem \ref{m} (i), we deduce that 
\begin{equation*}
	\lim\limits_{t\to +\infty} (u(x,t),v(x,t)) = (0,1)~\text{uniformly~for}~x\in \Omega.
\end{equation*}
The numerical experiment result is shown in Figure \ref{fig1}  (a).

If we choose $a_1=2$, $b_1=c_1=1$, $a_2=1$, $b_2=1$ and $c_2=2$, then $\frac{b_1}{b_2}< \frac{a_1}{a_2}$ , $\frac{c_1}{c_2}< \frac{a_1}{a_2}$, $\frac{a_1}{b_1}=2$. Thus, by Theorem \ref{m} (ii), we deduce that 
\begin{equation*}
	\lim\limits_{t\to +\infty} (u(x,t),v(x,t)) = (2,0)~\text{uniformly~for}~x\in \Omega.
\end{equation*}
The numerical experiment result is shown in Figure \ref{fig1}  (b).
\begin{figure*}[!t]
	\centering
	\subfigure[$v$ beats $u$] {\includegraphics[height=2in,width=3in,angle=0]{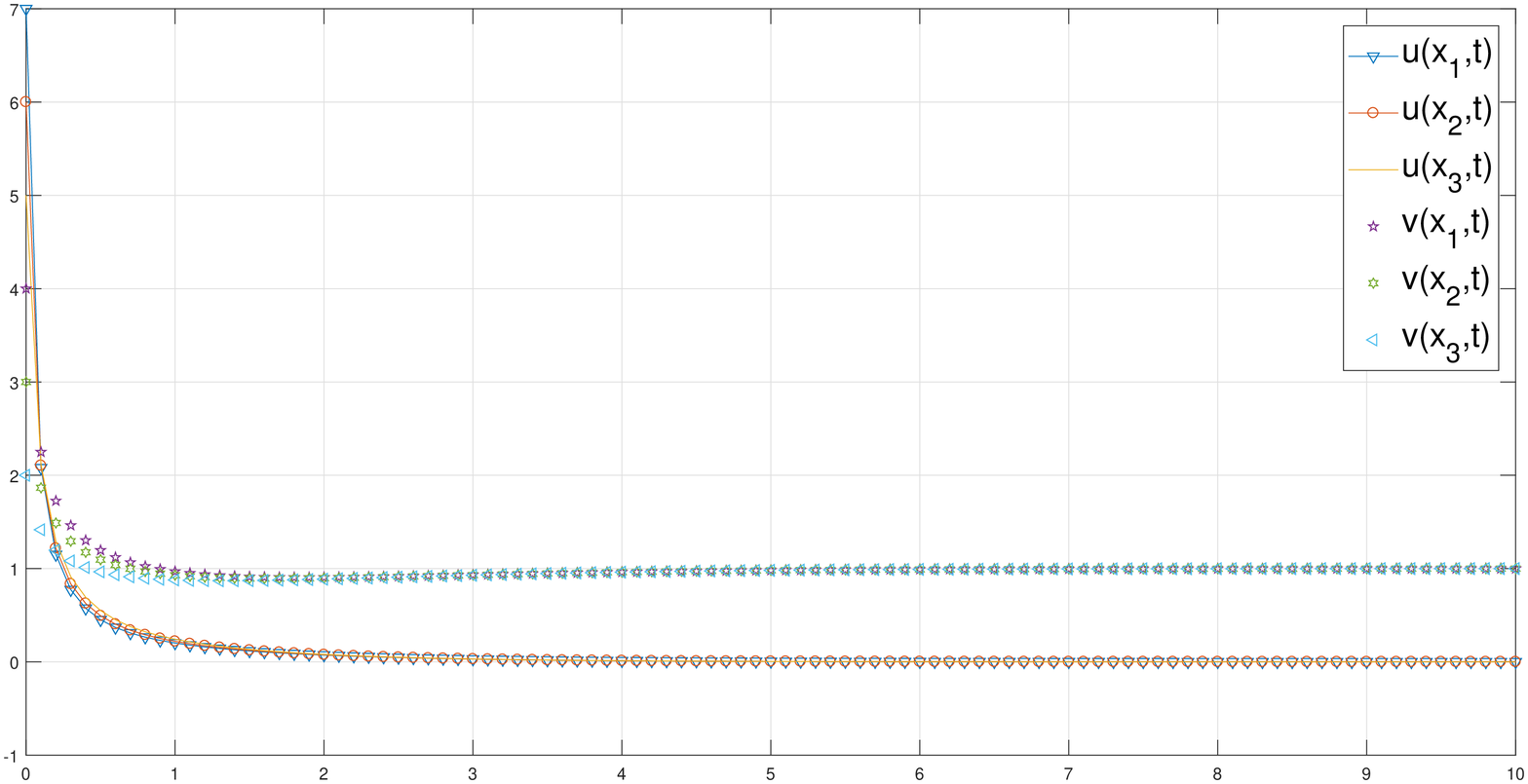}}
	\subfigure[$u$ beats $v$] {\includegraphics[height=2in,width=3in,angle=0]{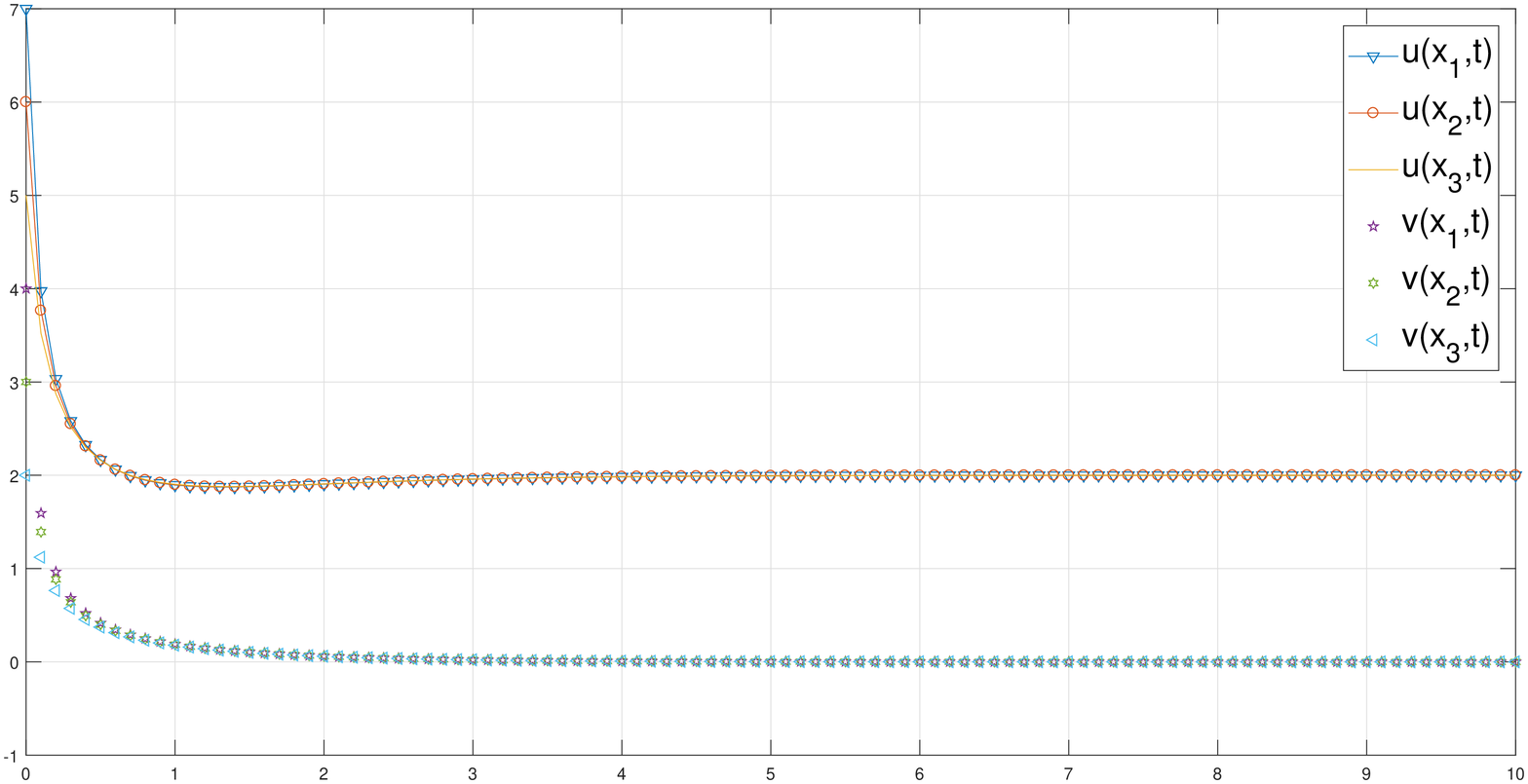}}
	\caption{}
	\label{fig1}
\end{figure*}

If we choose $a_1=2$, $b_1=c_1=1$, $a_2=3$, $b_2=1$ and $c_2=2$, then $\xi=1$, $\eta=1$, $	\frac{c_{1}}{c_{2}}<\frac{a_{1}}{a_{2}}<\frac{b_{1}}{b_{2}}$. Thus, by Theorem \ref{m} (iii), we deduce that 
\begin{equation*}
	\lim\limits_{t\to +\infty} (u(x,t),v(x,t)) = (1,1)~\text{uniformly~for}~x\in \Omega.
\end{equation*}
The numerical experiment result is shown in Figure \ref{fig2}  (a).

\begin{figure*}[!t]
	\centering
	\subfigure[coexistence] {\includegraphics[height=2in,width=3in,angle=0]{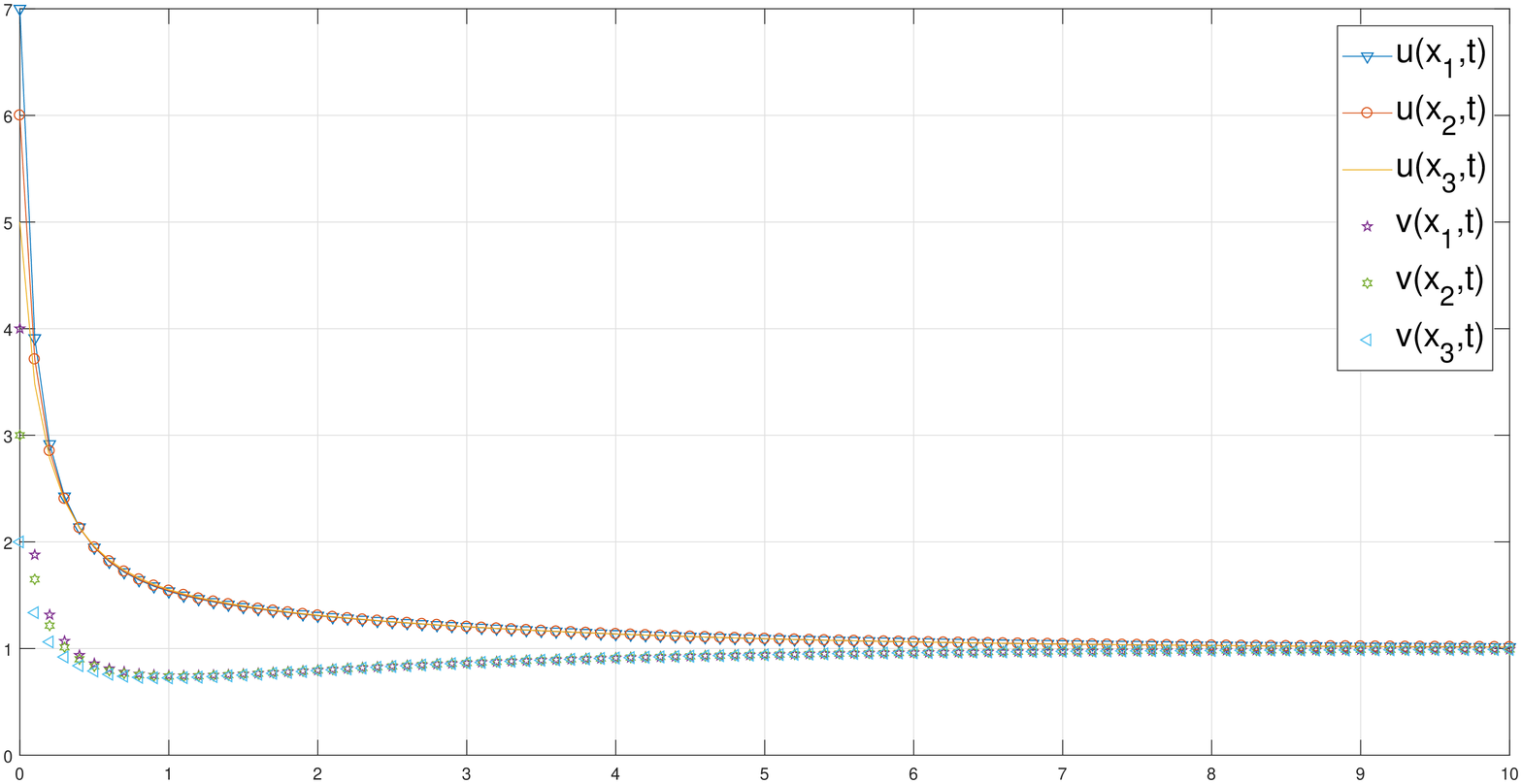}}
	\subfigure[$u$ beats $v$] {\includegraphics[height=2in,width=3in,angle=0]{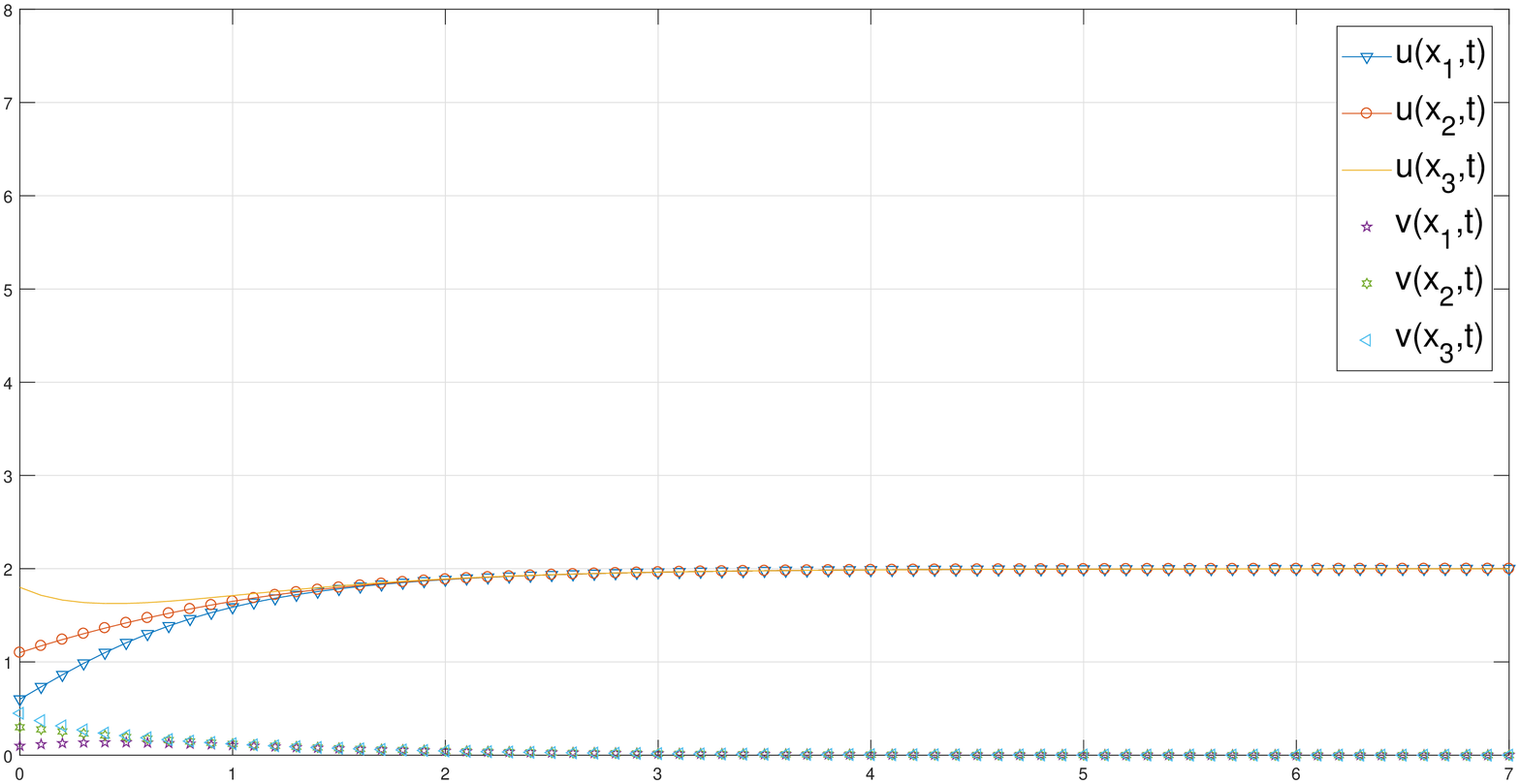}}
	\caption{}
	\label{fig2}
\end{figure*}

Suppose that $(u(x,t),v(x,t))$ is the unique global solution of \eqref{20220904-3} with $a_1=2$, $b_1=1,~c_1=3$, $a_2=1$, $b_2=1$ and $c_2=1$. 

If we choose $u_0(x_1)=0.6$, $u_0(x_2)=1.1$, $u_0(x_3)=1.8$, $v_0(x_1)=0.1$, $v_0(x_2)=0.3$, $v_0(x_3)=0.45$, then $\frac{b_1}{b_2}<\frac{a_1}{a_2}< \frac{c_1}{c_2}$, $0.5=\xi<u_{0}<\frac{a_1}{b_1}=2$, $0<v_{0}<0.5=\eta$ on $\bar{ \Omega}$. Thus, by Theorem \ref{m} (iv), we deduce that \begin{equation*}
	\lim\limits_{t\to +\infty} (u(x,t),v(x,t)) = (2,0)~\text{uniformly~for}~x\in \Omega.
\end{equation*}
The numerical experiment result is shown in Figure \ref{fig2}  (b).
If we choose $u_0(x_1)=0.1$, $u_0(x_2)=0.3$, $u_0(x_3)=0.4$, $v_0(x_1)=0.6$, $v_0(x_2)=0.78$, $v_0(x_3)=0.9$, then $\frac{b_1}{b_2}<\frac{a_1}{a_2}< \frac{c_1}{c_2}$, $0<u_{0}<0.5=\xi$, $0.5=\eta<v_{0}<\frac{a_2}{c_2}=1$ on $\bar{ \Omega}$. Thus, by Theorem \ref{m} (iv), we deduce that \begin{equation*}
	\lim\limits_{t\to +\infty} (u(x,t),v(x,t)) = (0,1)~\text{uniformly~for}~x\in \Omega.
\end{equation*}
The numerical experiment result is shown in Figure \ref{fig3}  (a).

\begin{example}\label{61}
	Choose a graph $G=(V,E)$ satisfying $V=\{x_1 ,x_2, x_3 \}$ whose vertices are linked as the following figure with a weight $\omega$ satisfying 
	$
	\omega_{x_i x_j}=\begin{cases}
		1,~x_i \sim x_j,\\
		0,~x_i \not= x_j
	\end{cases}
	$
	for $i,j=1,2,3$. Let $d_1=d_2=1$ and $\mu(x)=\sum\limits_{y\in V:y\sim x} \omega_{xy}$.
	\begin{center}
		\begin{tikzpicture}[>=stealth, scale=1.2]		
			\draw(1,0)node{$\bullet$}--(2,0)node{$\bullet$};
			\draw(1,0)--(1.5,0.867)node{$\bullet$}--(2,0);
			\draw(1,0)node[below]{\hspace{-.1cm}$x_1$}(2,0)node[below]{\hspace{.3cm}$x_3$} (1.5,0.867)node[above]{$x_2$};
		\end{tikzpicture}
	\end{center}

\end{example}	
	
	Suppose that $(u(x,t),v(x,t))$ is the unique global positive solution of \eqref{41} with 
	$$\Delta_{V_{1}}u(x,t)=\sum_{y\in V}[u(y,t)-u(x,t)] \frac{\omega_{xy}}{\mu(x)},$$
	$$\Delta_{V_{2}}v(x,t)=\sum_{y\in V}[v(y,t)-v(x,t)] \frac{\omega_{xy}}{\mu(x)},$$
	 $u_0(x_1)=7$, $u_0(x_2)=6$, $u_0(x_3)=5$, $v_0(x_1)=4$, $v_0(x_2)=3$, and $v_0(x_3)=2$.

If we choose $a_1=1$, $b_1=c_1=2$, $a_2=1$, $b_2=1$ and $c_2=1$, then $\frac{a_1}{a_2}< \frac{b_1}{b_2}$ , $\frac{a_1}{a_2}< \frac{c_1}{c_2}$, $\frac{a_2}{c_2}=1$. Thus, by Theorem \ref{m1} (i), we deduce that 
\begin{equation*}
	\lim\limits_{t\to +\infty} (u(x,t),v(x,t)) = (0,1)~\text{uniformly~for}~x\in V.
\end{equation*}
The numerical experiment result is shown in Figure \ref{fig3}  (b).
\begin{figure*}[!t]
	\centering
	\subfigure[$v$ beats $u$] {\includegraphics[height=2in,width=3in,angle=0]{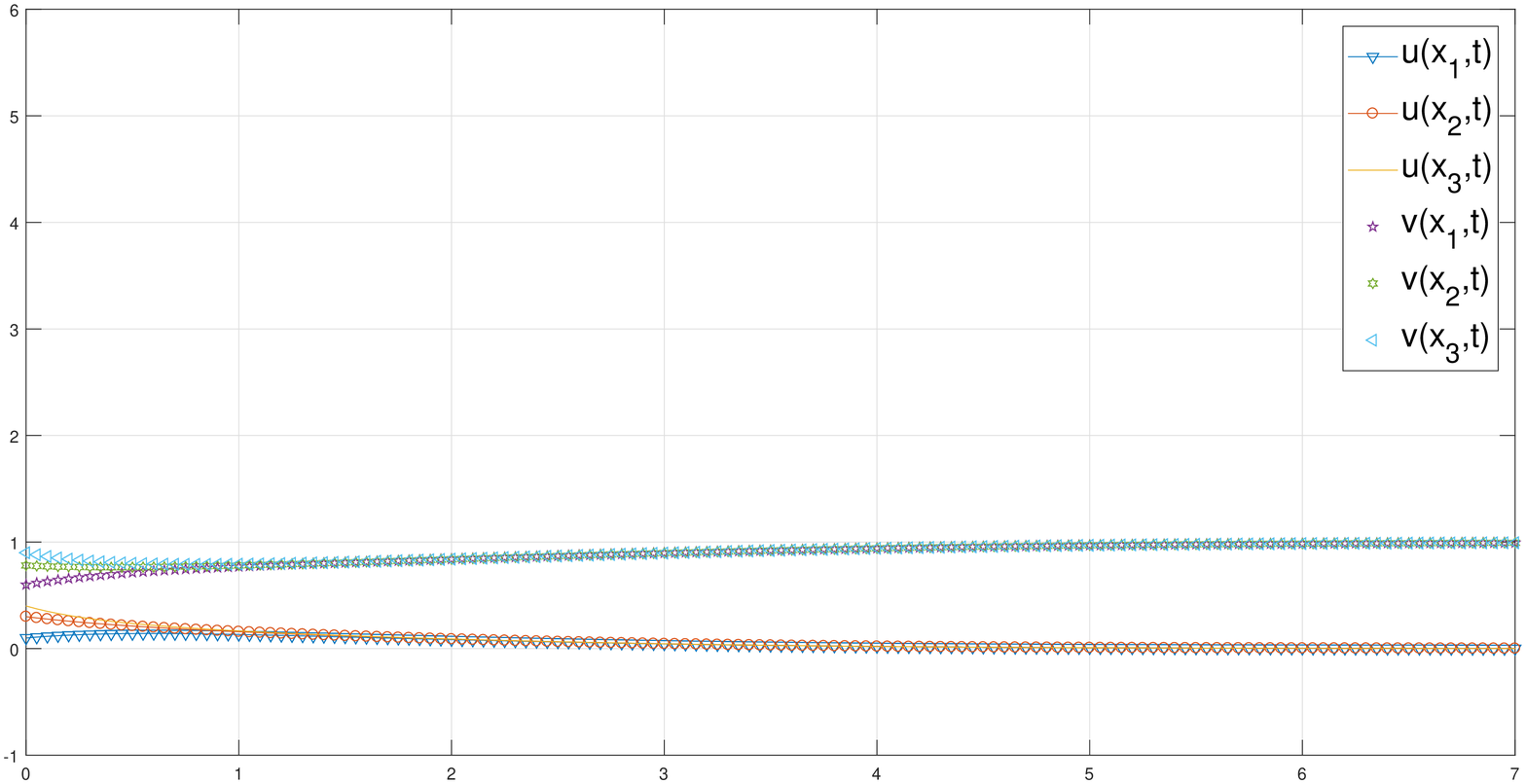}}
	\subfigure[$u$ beats $v$] {\includegraphics[height=2in,width=3in,angle=0]{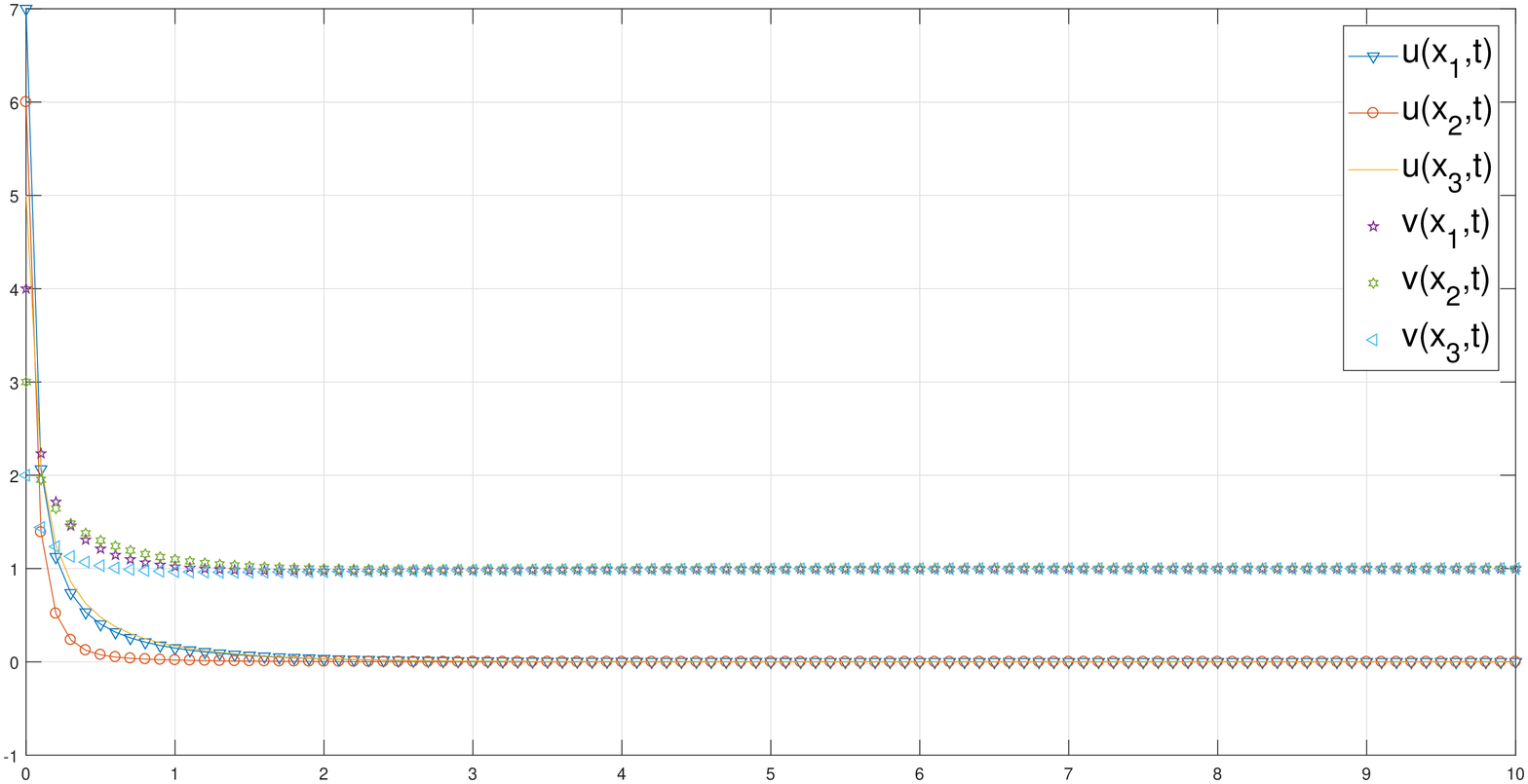}}
	\caption{}
	\label{fig3}
\end{figure*}

If we choose $a_1=2$, $b_1=c_1=1$, $a_2=1$, $b_2=1$ and $c_2=2$, then $\frac{b_1}{b_2}< \frac{a_1}{a_2}$ , $\frac{c_1}{c_2}< \frac{a_1}{a_2}$, $\frac{a_1}{b_1}=2$. Thus, by Theorem \ref{m1} (ii), we deduce that 
\begin{equation*}
	\lim\limits_{t\to +\infty} (u(x,t),v(x,t)) = (2,0)~\text{uniformly~for}~x\in V.
\end{equation*}
The numerical experiment result is shown in Figure \ref{fig4}  (a).

If we choose $a_1=2$, $b_1=c_1=1$, $a_2=3$, $b_2=1$ and $c_2=2$, then $\xi=1$, $\eta=1$, $	\frac{c_{1}}{c_{2}}<\frac{a_{1}}{a_{2}}<\frac{b_{1}}{b_{2}}$. Thus, by Theorem \ref{m1} (iii), we deduce that 
\begin{equation*}
	\lim\limits_{t\to +\infty} (u(x,t),v(x,t)) = (1,1)~\text{uniformly~for}~x\in V.
\end{equation*}
The numerical experiment result is shown in Figure \ref{fig4}  (b).

\begin{figure*}[!t]
	\centering
	\subfigure[$u$ beats $v$] {\includegraphics[height=2in,width=3in,angle=0]{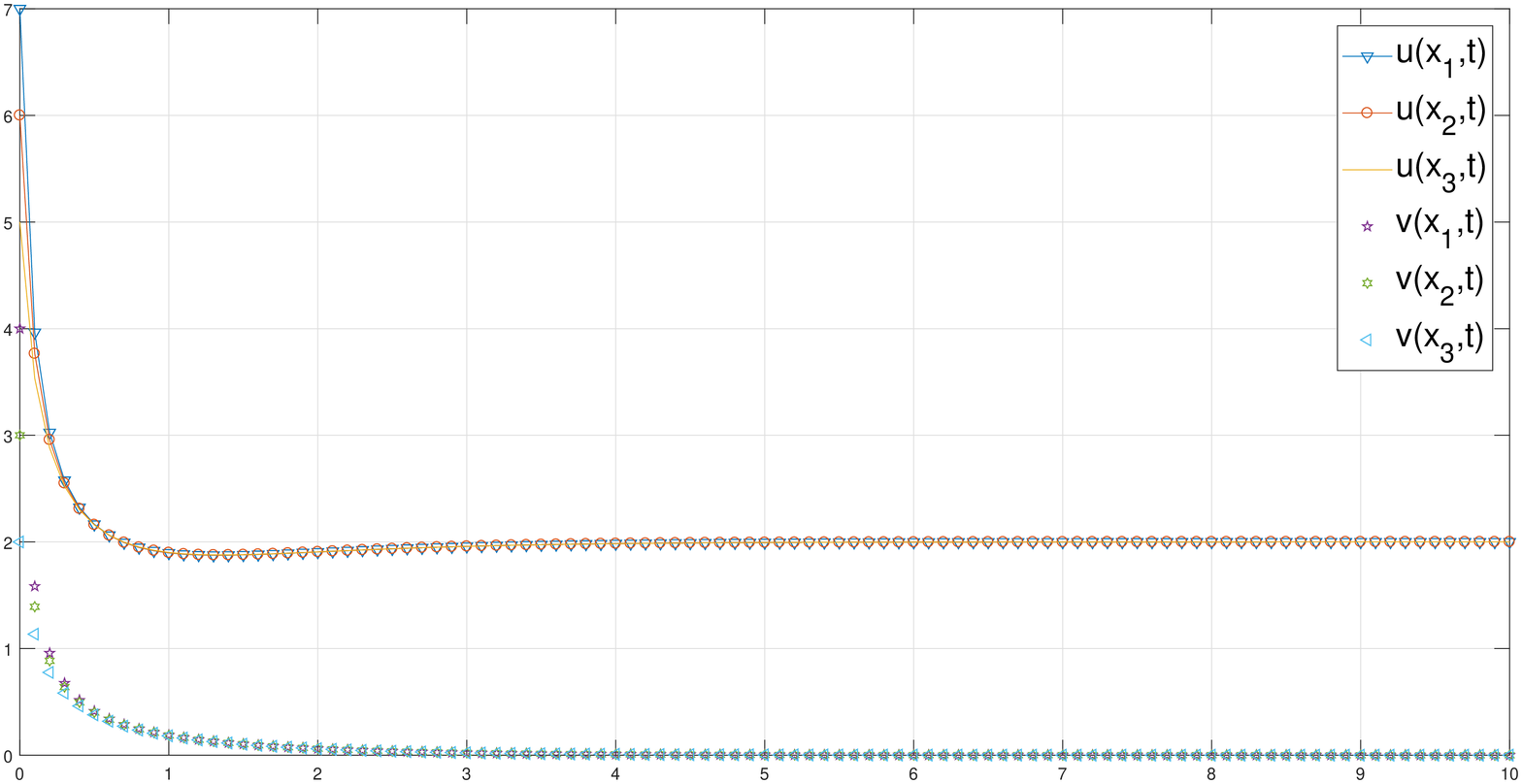}}
	\subfigure[coexistence] {\includegraphics[height=2in,width=3in,angle=0]{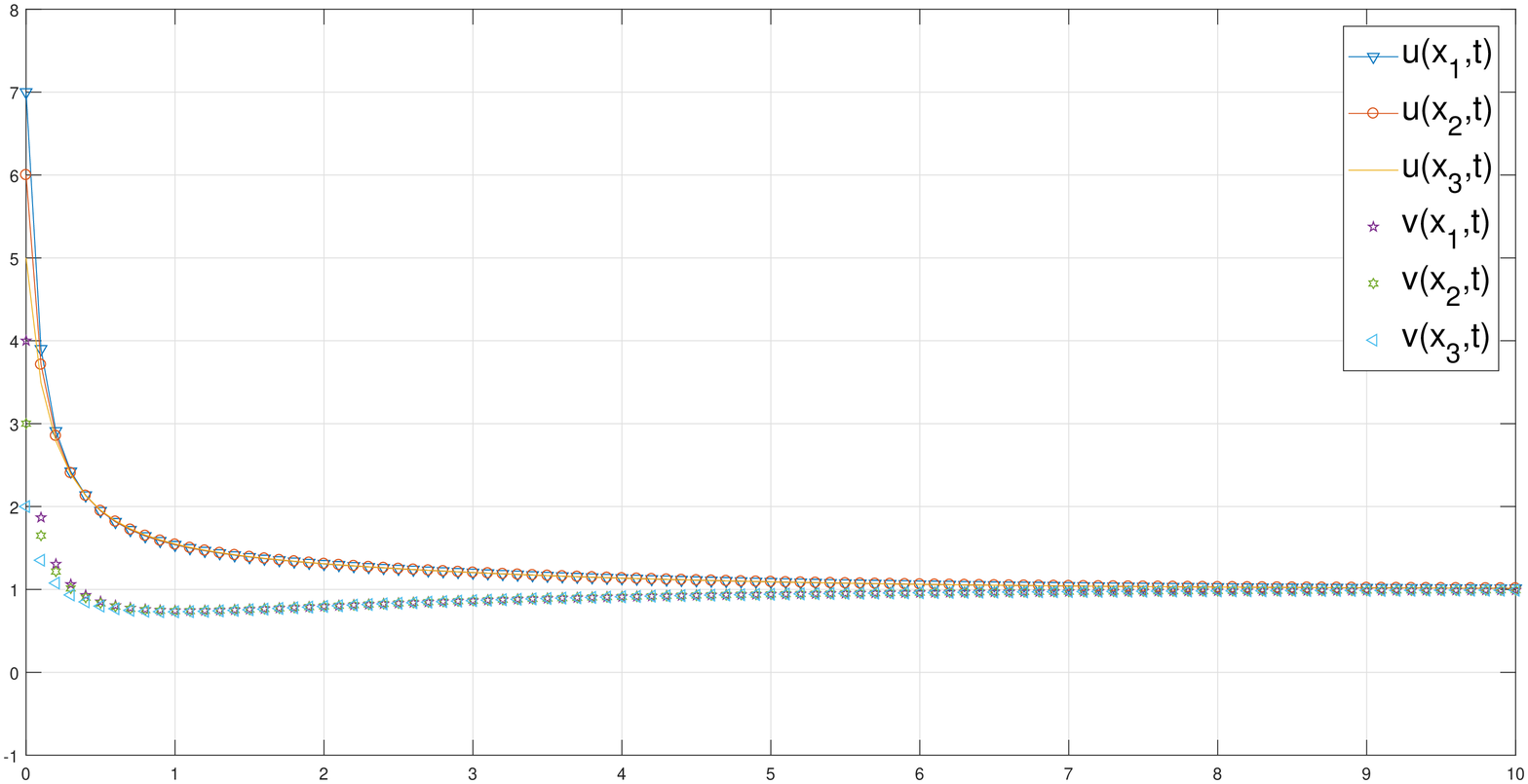}}
	\caption{}
	\label{fig4}
\end{figure*}
\begin{figure*}[!t]
	\centering
	\subfigure[$u$ beats $v$] {\includegraphics[height=2in,width=3in,angle=0]{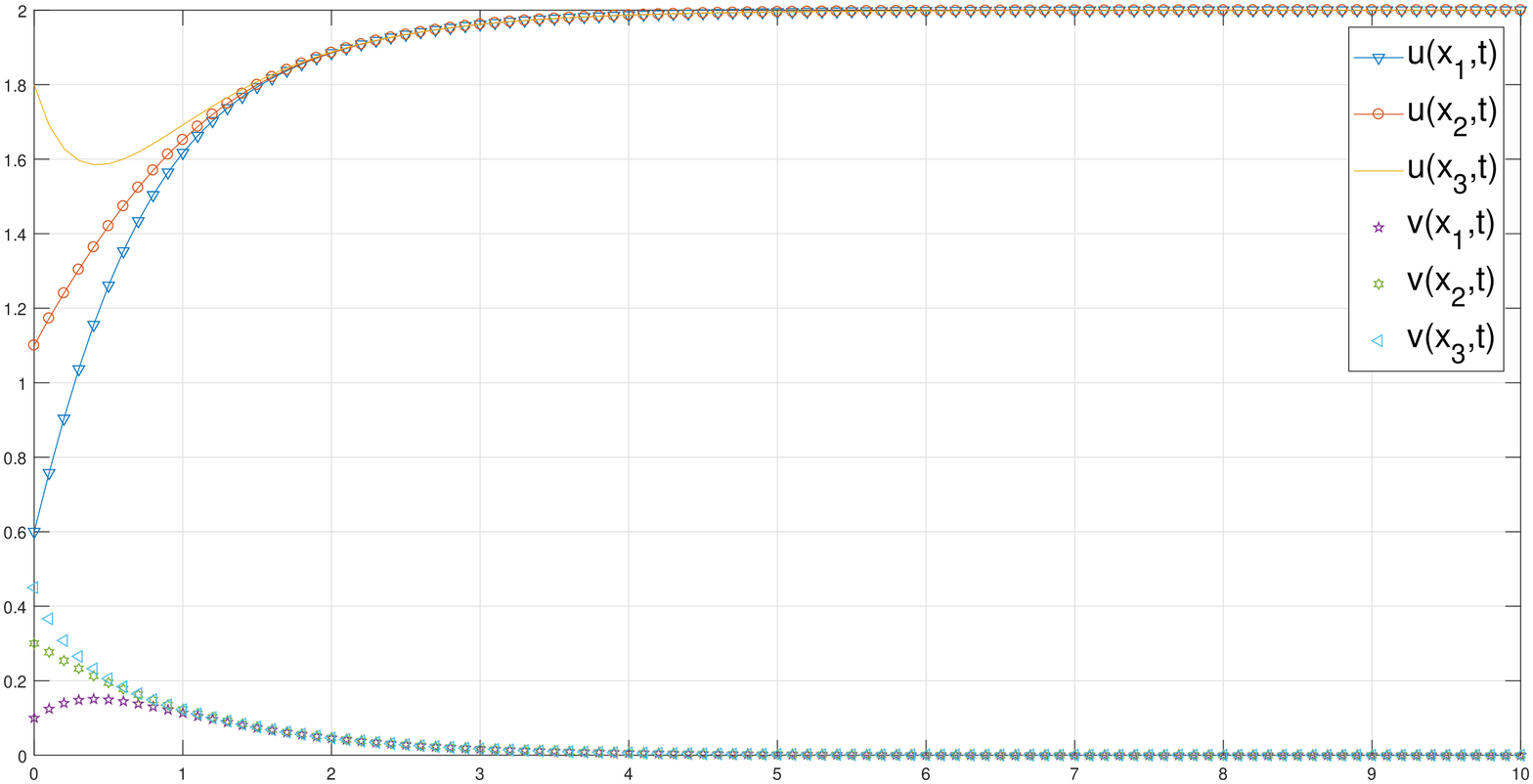}}
	\subfigure[$u$ beats $v$] {\includegraphics[height=2in,width=3in,angle=0]{th24.eps}}
	\caption{}
	\label{fig5}
\end{figure*}

Suppose that $(u(x,t),v(x,t))$ is the unique global positive solution of \eqref{20220904-3} with $a_1=2$, $b_1=1,~c_1=3$, $a_2=1$, $b_2=1$ and $c_2=1$. 

If we choose $u_0(x_1)=0.6$, $u_0(x_2)=1.1$, $u_0(x_3)=1.8$, $v_0(x_1)=0.1$, $v_0(x_2)=0.3$, $v_0(x_3)=0.45$, then $\frac{b_1}{b_2}<\frac{a_1}{a_2}< \frac{c_1}{c_2}$, $0.5=\xi<u_{0}<\frac{a_1}{b_1}=2$, $0<v_{0}<0.5=\eta$ on $V$. Thus, by Theorem \ref{m1} (iv), we deduce that \begin{equation*}
	\lim\limits_{t\to +\infty} (u(x,t),v(x,t)) = (2,0)~\text{uniformly~for}~x\in V.
\end{equation*}
The numerical experiment result is shown in Figure \ref{fig5}  (a).
If we choose $u_0(x_1)=0.1$, $u_0(x_2)=0.3$, $u_0(x_3)=0.4$, $v_0(x_1)=0.6$, $v_0(x_2)=0.78$, $v_0(x_3)=0.9$, then $\frac{b_1}{b_2}<\frac{a_1}{a_2}< \frac{c_1}{c_2}$, $0<u_{0}<0.5=\xi$, $0.5=\eta<v_{0}<\frac{a_2}{c_2}=1$ on $V$. Thus, by Theorem \ref{m1} (iv), we deduce that \begin{equation*}
	\lim\limits_{t\to +\infty} (u(x,t),v(x,t)) = (0,1)~\text{uniformly~for}~x\in V.
\end{equation*}
The numerical experiment result is shown in Figure \ref{fig5}  (b).

\end{document}